\documentclass[11pt,letterpaper]{amsart}
\usepackage{dslihead}

\title{The plectic conjecture over function fields}

\author{Siyan Daniel Li-Huerta}
\email{sli@math.harvard.edu}
\address{Department of Mathematics\\Harvard University\\1 Oxford Street\\Cambridge, MA 02138}


\usepackage{amsmath}

\swapnumbers
\theoremstyle{plain}
\newtheorem*{thm*}{Theorem}
\newtheorem*{lem*}{Lemma}

\newtheorem*{prop*}{Proposition}
\newtheorem*{conj*}{Conjecture}

\newtheorem*{thmA}{Theorem A}
\newtheorem*{thmB}{Theorem B}
\newtheorem*{thmC}{Theorem C}

\theoremstyle{definition}

\newtheorem*{defn*}{Definition}

\theoremstyle{remark}
\newtheorem{rem}[subsection]{Remark}
\newtheorem*{rem*}{Remark}
\newtheorem*{rems*}{Remarks}

\begin{document}

\begin{abstract}
We prove the plectic conjecture of Nekov\'a\v{r}--Scholl \cite{NS16} over global function fields $Q$. For example, when the cocharacter is defined over $Q$ and the structure group is a Weil restriction from a geometric degree $d$ separable extension $F/Q$, consider the complex computing $\ell$-adic intersection cohomology with compact support of the associated moduli space of shtukas over $Q_I$. We endow this with the structure of a complex of $(\Weil(F)^d\rtimes\fS_d)^I$-modules, which extends its structure as a complex of $\Weil(Q)^I$-modules constructed by Arinkin--Gaitsgory--Kazhdan--Raskin--Rozenblyum--Varshavsky. We show that the action of $(\Weil(F)^d\rtimes\fS_d)^I$ commutes with the Hecke action, and we give a moduli-theoretic description of the action of Frobenius elements in $\Weil(F)^{d\times I}$.
\end{abstract}

\maketitle
\tableofcontents

\section*{Introduction}
The \emph{plectic conjecture} of Nekov\'a\v{r}--Scholl \cite{NS16} predicts extra symmetries in the cohomology of Shimura varieties when the structure group $G$ is a Weil restriction. In the $\ell$-adic realization, the case of trivial coefficients is formulated as follows. Suppose that $G$ is the Weil restriction $\R_{F/\bQ}H$ of a connected reductive group $H$ over a number field $F$. We have the \emph{plectic Galois group} $\Ga^{\plec}_{F/\bQ}\coloneqq\Aut_F(F\otimes_\bQ\ov\bQ)$, which naturally admits a continuous injective homomorphism from the absolute Galois group $\Ga_\bQ$ of $\bQ$. The plectic Galois group acts naturally on the set of conjugacy classes of cocharacters of $G_{\ov\bQ}$, and we can form the stabilizer $\Ga^{[\mu]}_{F/\bQ}$ of the Hodge cocharacter $[\mu]$ in $\Ga^{\plec}_{F/\bQ}$. Note that the reflex field $E$ is characterized by $\Ga_E = \Ga_\bQ\cap\Ga^{[\mu]}_{F/\bQ}$. For sufficiently large level $N$, write $\ov{\Sh}_N$ for the minimal compactification of our Shimura variety at level $N$ over $E$.
\begin{conj*}[{\cite[Conjecture 6.1]{NS16}}]
The intersection cohomology complex of $\ov{\Sh}_N$ with coefficients in $\ov\bQ_\ell$ canonically lifts from an object of $D^b(\Ga_E,\ov\bQ_\ell)$ to an object of $D^b(\Ga^{[\mu]}_{F/\bQ},\ov\bQ_\ell)$.
\end{conj*}
At the time of writing, the plectic conjecture in the number field setting is wide open. The goal of this paper is to prove the plectic conjecture in the function field setting. More precisely, we prove that an analogous phenomenon holds for the \emph{moduli space of shtukas}, which is an equi-characteristic analogue of Shimura varieties. However, moduli spaces of shtukas admit richer variants than their number field counterparts: namely, the ability to have multiple \emph{legs}. This already plays a crucial role in applications to the Langlands program \cite{Laf16, YZ17}, and it also plays a crucial role in this paper.

To state our results, we need some notation. Let $Q$ be a global field of positive characteristic, write $k=\bF_q$ for its constant field, and assume that $\ell\nmid q$. Henceforth let $F$ be a degree $d$ separable extension of $Q$ with the same constant field, let $H$ be a connected reductive group over $F$, and write $G$ for the Weil restriction $\R_{F/Q}H$. Let $I$ be a finite set, and let $\ul\om=(\om_i)_{i\in I}$ be an $I$-tuple of conjugacy classes of cocharacters of $G_{\ov{Q}}$ such that each $\om_i$ is defined over $Q$.\footnote{We can always enlarge $Q$ such that the $\om_i$ are defined over $Q$. This is analogous to the number field setting, since the field of definition of $[\mu]$ is precisely the field over which our Shimura variety lives.} Write $X$ for the geometrically connected smooth proper curve over $k$ associated with $Q$, and write $Q_I$ for the generic point of $X^I$. For any finite closed subscheme $N$ of $X$, we get a moduli space of shtukas $\Sht_{G,N,I,\ul\om}|_{Q_I}$ at level $N$ over $Q_I$.\footnote{Strictly speaking, we need to choose a parahoric group scheme over $X$ with generic fiber $G$. We also need to choose an ordered partition of $I$, in order to define partial Frobenius morphisms. However, we will ignore these issues for the rest of the introduction.} Work of Xue \cite[Proposition 6.0.10]{Xue20b} yields a natural $\Weil(Q)^I$-action on the intersection cohomology groups with compact support of $\Sht_{G,N,I,\ul\om}|_{Q_I}$, and forthcoming work of Arinkin--Gaitsgory--Kazhdan--Raskin--Rozenblyum--Varshavsky enhances this $\Weil(Q)^I$-action to the level of complexes.

We now turn to the plectic group in our setting. Since the $\om_i$ are defined over $Q$, they are stabilized by all of $\Ga^{\plec}_{F/Q}$. By fixing extensions of the $d$ different $Q$-embeddings $F\hookrightarrow\ov{Q}$ to automorphisms of $\ov{Q}$ over $Q$, we can identify $\Ga^{\plec}_{F/Q}$ with the semidirect product $\Ga_Q^d\rtimes\fS_d$, where $\fS_d$ denotes the $d$-th symmetric group. Applying similar observations to the Weil group yields a continuous injective homomorphism $\Weil(Q)\hookrightarrow\Weil(F)^d\rtimes\fS_d$.
\begin{thmA}
The complex of intersection cohomology with compact support of $\Sht_{G,N,I,\ul\om}|_{Q_I}$ with coefficients in $\ov\bQ_\ell$ canonically lifts from an object of $D^b(\Weil(Q)^I,\ov\bQ_\ell)$ to an object of $D^b((\Weil(F)^d\rtimes\fS_d)^I,\ov\bQ_\ell)$.
\end{thmA}
Because shtukas can have multiple legs, the power of $I$ appears in Theorem A. However, even when $I$ is a singleton (which mirrors the plectic conjecture in the number field setting), the ability to have multiple legs plays a crucial role in the proof of the plectic conjecture in the function field setting.
\begin{rem*}
For $\ul\om$ not necessarily defined over $Q$, our methods prove a similar result for the intersection cohomology with compact support of a union of the plectic Galois translates of $\Sht_{G,N,I,\ul\om}$. In fact, all our results apply in this level of generality. See Theorem \ref{ss:maintheorem}, Theorem \ref{ss:plecticalgebra}, and Theorem \ref{ss:plecticfrob}.
\end{rem*}
The $(\Weil(F)^d\rtimes\fS_d)^I$-action we construct enjoys the following compatibility. Write $\fH_{G,N}$ for the Hecke algebra of $G$ at level $N$, which acts naturally on $\Sht_{G,N,I,\ul\om}|_{Q_I}$ via finite \'etale correspondences and hence on its intersection cohomology groups with compact support.
\begin{thmB}
The action of $(\Weil(F)^d\rtimes\fS_d)^I$ from Theorem A on the level of cohomology groups commutes with the action of $\fH_{G,N}$.
\end{thmB}
We can also describe the action of Frobenius elements in $(\Weil(F)^d\rtimes\fS_d)^I$ in terms of \emph{partial Frobenius morphisms}, as conjectured in \cite[Remark 6.7]{NS16}. Now $F/Q$ corresponds to a finite morphism $m:Y\ra X$, where $Y$ is also geometrically connected over $k$. Let $k'$ be a degree $r$ extension of $k$, and let $\ul{x}=(x_i)_{i\in I}$ be a $k'$-point of $X^I$ such that each $x_i$ \emph{splits completely} in $Y$, i.e. $m^{-1}(x_i)$ is a disjoint union of $k'$-points $(y_{h,i})_{h=1}^d$. For $x_i$ lying in a certain dense open subscheme $U\ssm N$ of $X$, a smoothness result of Xue \cite[Theorem 6.0.12]{Xue20b} identifies the intersection cohomology groups with compact support of $\Sht_{G,I,\ul\om}|_{Q_I}$ and the intersection cohomology groups with compact support of $\Sht_{G,I,\ul\om}|_{\ul{x}}$. Write $\ul{y}$ for the $k'$-point $(y_{h,i})_{(h,i)\in d\times I}$ of $Y^{d\times I}$. Diagrams (\ref{eq:pl}) and (\ref{eq:sym}) below will enable us to identify $\Sht_{G,I,\ul\om}|_{\ul{x}}$ and $\Sht_{H,d\times I,\ul\om}|_{\ul{y}}$ up to universal homeomorphism.

We now introduce partial Frobenii. Write $V$ and $M$ for the preimages of $U$ and $N$ in $Y$. For any $(h,i)$ in $d\times I$, we have a commutative square
\begin{align*}
\xymatrix{\Sht_{H,d\times I,\ul\om}|_{(V\ssm M)^{d\times I}}\ar[r]^-{\Fr_{(h,i)}}\ar[d]^-\fp & \Sht_{H,d\times I,\ul\om}|_{(V\ssm M)^{d\times I}}\ar[d]^-\fp\\
(V\ssm M)^{d\times I}\ar[r]^-{\Frob_{(h,i)}} & (V\ssm M)^{d\times I},
}
\end{align*}
where $\Frob_{(h,i)}$ equals absolute $q$-Frobenius on the $(h,i)$-th factor and the identity on the other factors. Therefore $\Fr_{(h,i)}$ induces a $\Frob_{(h,i)}$-semilinear endomorphism $F_{(h,i)}$ of the relative intersection cohomology with compact support of $\Sht_{H,d\times I,\ul\om}|_{(V\ssm M)^{d\times I}}$ over $(V\ssm M)^{d\times I}$. As $\Frob_{(h,i)}^r$ fixes $\ul{y}$, we obtain an action of $F_{(h,i)}^r$ on the intersection cohomology groups with compact support of $\Sht_{H,d\times I,\ul\om}|_{\ul{y}}$.

On the other hand, we also have Frobenius elements in Weil groups. Namely, the $k'$-point $y_{h,i}$ of $Y$ yields a geometric $q^r$-Frobenius element $\ga_{y_{h,i}}$ in $\Weil(F)$, which acts on the intersection cohomology groups with compact support of $\Sht_{G,I,\ul\om}|_{Q_I}$ via the $(h,i)$-th factor of $\Weil(F)^{d\times I}$ in Theorem A.
\begin{thmC}
Under these identifications, the action of $\ga_{y_{h,i}}$ equals the action of $F^r_{(h,i)}$.
\end{thmC}
Let us now discuss the proofs of our theorems. For simplicity, assume that $H$ is split, take $N=\varnothing$, and suppose that $F$ is everywhere unramified over $Q$.\footnote{We treat the general case in the body of the paper, and this simplified case already illustrates the main ideas.} Thus $m:Y\ra X$ is \'etale.

We begin by observing that $G$-bundles on $X$ are naturally equivalent to $H$-bundles on $Y$. Moreover, this equivalence is compatible with replacing $X$ by the punctured curve $X\ssm x$, as long as $Y$ is replaced by $Y\ssm m^{-1}(x)$. We use this to show the existence of a Cartesian square
\begin{align}\label{eq:pl}
\xymatrix{\Sht_{G,I,\ul\om}\ar[r]\ar[d]^-\fp & \Sht^{(d)}_{H,I,\ul\om}\ar[d]^-\fp\\
X^I\ar[r]^-{(m^{-1})^I} & (\Div^d_Y)^I,
}\tag{$\triangleright$}
\end{align}
where $\Div^d_Y$ denotes the space of degree $d$ divisors of $Y$, and $\Sht^{(d)}_{H,I,\ul\om}$ denotes a \emph{symmetrized} variant of the moduli space of shtukas that keeps track of an $I$-tuple of \emph{divisors} of $Y$, instead of just points of $Y$. We make important use of this symmetrized variant, so we study it thoroughly in \S\ref{s:shtukas}. Diagram (\ref{eq:pl}) provides one incarnation of the conjectured \emph{plectic diagram} from \cite[(1.3)]{NS16}. Because $m$ is \'etale, the image of the closed immersion $m^{-1}:X\ra \Div^d_Y$ lies in the open subscheme $\Div^{d,\circ}_Y$ of \'etale divisors. 

We can relate $\Sht^{(d)}_{H,I,\ul\om}$ to a usual, unsymmetrized moduli space of shtukas as follows. By viewing $\Div^d_Y$ as the scheme-theoretic quotient of $Y^d$ by $\fS_d$, we get a commutative square
\begin{align}\label{eq:sym}
\xymatrix{\Sht^{(d)}_{H,I,\ul\om}\ar[d]^-\fp &\ar[l] \Sht_{H,d\times I,\ul\om}\ar[d]^-\fp\\
(\Div_Y^d)^I&\ar[l]_-\al Y^{d\times I}
}\tag{$\triangleleft$}
\end{align}
that is Cartesian up to universal homeomorphism, where we use $G=\R_{F/Q}H$ to view each $\om_i$ as a $d$-tuple of conjugacy classes of cocharacters of $H_{\ov{F}}$. Note that $\fS_d^I$ acts naturally on the right-hand side. Now Arinkin--Gaitsgory--Kazhdan--Raskin--Rozenblyum--Varshavsky's result endows the complex of intersection cohomology with compact support of $\Sht_{H,d\times I,\ul\om}|_{F_{d\times I}}$ with the structure of a complex of $\Weil(F)^{d\times I}$-modules, and the $\fS_d^I$-action intertwines the $\Weil(F)^{d\times I}$-action by permutation. We use this to obtain the structure of a complex of $\Weil(F)^{d\times I}\rtimes\fS_d^I = (\Weil(F)^d\rtimes\fS_d)^I$-modules.

By applying proper base change to Diagrams (\ref{eq:pl}) and (\ref{eq:sym}), Xue's smoothness result \cite[Theorem 4.2.3]{Xue20b} identifies the complexes of intersection cohomology with compact support of $\Sht_{G,I,\ul\om}|_{Q_I}$ and intersection cohomology with compact support of $\Sht_{H,d\times I,\ul\om}|_{F_{d\times I}}$. Under this identification, we check that the action of $\Weil(Q)^I$ agrees with the action of its image in $(\Weil(F)^d\rtimes\fS_d)^I$, which completes the proof of Theorem A. From here, we deduce Theorem B by generalizing Hecke correspondences to $\Sht^{(d)}_{H,I,d\times\ul\om}$ and showing that they are compatible with Diagrams (\ref{eq:pl}) and (\ref{eq:sym}). Finally, we obtain Theorem C using the fact that the $\Weil(F)^{d\times I}$-action on the intersection cohomology groups with compact support of $\Sht_{H,d\times I,\ul\om}|_{F_{d\times I}}$ is constructed by applying \emph{Drinfeld's lemma} to the $F_{(h,i)}$. Now Drinfeld's lemma does not immediately apply, as these cohomology groups are not finite-dimensional over $\ov\bQ_\ell$, but we use results of Xue to circumvent this. Note that even when $I$ is a singleton, $d\times I$ usually is not, so Drinfeld's lemma and therefore multiple-leg phenomena play a crucial role in this paper.

\begin{rems*}\hfill
  \begin{enumerate}[(1)]
  \item We appeal to forthcoming work of Arinkin--Gaitsgory--Kazhdan--Raskin--Rozenblyum--Varshavsky for two reasons: to obtain results for general $U$ and $N$, and to obtain results on the level of complexes. If we only want Theorem A when $U\ssm N=X$, then we only need existing results from \cite{AGKRRV21}. If we only want Theorem A on the level of cohomology groups, then we only need results of Xue \cite[Proposition 6.0.10]{Xue20b} instead.

\item We require that $F$ has the same constant field as $Q$ in order for \cite{Xue20b} and \cite{AGKRRV21} to apply to the moduli of shtukas over $Y$. Without this hypothesis, $Y^{d\times I}$ may be disconnected, so its local systems are no longer dictated by representations of a single group. However, we expect some version of \cite{Xue20b} and \cite{AGKRRV21} to apply even without this hypothesis. Consequently, this would remove this hypothesis from Theorem A, Theorem B, and Theorem C.

\item Our strategy also applies to moduli spaces of local shtukas as in \cite{FS21}. In particular, we expect a proof of the plectic conjecture for local Shimura varieties on the level of complexes, which should yield applications to (global) Shimura varieties via uniformization. We hope to report on this soon.
  \end{enumerate}
\end{rems*}

Tamiozzo considered a variant of Diagram (\ref{eq:pl}) in his thesis, though he did not proceed further. After completing an earlier version of this paper, the author was informed that X. Zhu proposed a similar strategy for proving Theorem A, but only on the level of cohomology groups.

\subsection*{Outline}
In \S\ref{s:bung}, we collect facts on the moduli space of $G$-bundles, as well as certain relative variants thereof. In \S\ref{s:grass}, we introduce symmetrized versions of the Hecke stack and the Beilinson--Drinfeld Grassmannian, and we also recall the Beauville-Laszlo theorem and the geometric Satake correspondence. In \S\ref{s:shtukas}, we use the preceding material to define and study symmetrized versions of the moduli space of shtukas, which are the main characters of this paper. We also recall Xue's smoothness result here. In \S\ref{s:frob}, we discuss partial Frobenius morphisms, their relation to monodromy, and how they arise in the moduli space of shtukas. We also state the anticipated result of Arinkin--Gaitsgory--Kazhdan--Raskin--Rozenblyum--Varshavsky here. Finally, in \S\ref{s:plectic} we assemble everything and prove Theorem A, Theorem B, and Theorem C. We conclude by elaborating on a moduli-theoretic interpretation of Theorem C.

\subsection*{Notation} Unless otherwise specified, all fiber products and thus Cartesian powers are taken over $k$. We denote base changes with subscripts, possibly also with vertical restriction bars. For any connected algebraic stack $\cX$ over $k$, we always suppress base points and write $\pi_1(\cX)$ for the associated \'etale fundamental group. By a \emph{$G$-bundle}, we always mean a principal homogeneous space for $G$.

We view all derived categories as $\infty$-categories, and we interpret all operations on them $\infty$-categorically. For any locally profinite group $W$, write $D^b_c(W,\ov\bQ_\ell)$ for the bounded derived category of continuous finite-dimensional representations of $W$ over $\ov\bQ_\ell$, write $D(W,\ov\bQ_\ell)$ for its ind-completion, and write $D^b(W,\ov\bQ_\ell)\subseteq D(W,\ov\bQ_\ell)$ for the full subcategory of bounded objects. Finally, for any $\infty$-category $\cC$ with an action by a discrete group $H$, we write $\cC^{BH}$ for the $\infty$-category of $H$-equivariant objects in $\cC$.

\subsection*{Acknowledgments}
The author thanks Mark Kisin for his encouragement and patience. The author is especially indebted to Dennis Gaitsgory for explaining his work-in-progress, as well as Kevin Lin for answering countless questions about \cite{AGKRRV21}. The author would also like to thank Robert Cass and Tamir Hemo for helpful discussions.

\section{Moduli spaces of bundles}\label{s:bung}
In this section, we collect facts on the \emph{moduli space of $G$-bundles on $X$}, as it plays a central role in our discussion. We begin by fixing notation for our group schemes $G$ of interest over $X$, which serve as integral models for our structure group over $Q$. Then, we define the moduli space of $G$-bundles on $X$ with level structure, as well as certain relative variants which will be useful in \S\ref{s:grass}. We conclude by introducing Weil restrictions and how they affect $\Bun_G$, which is crucial for the results of this paper.

\subsection{}\label{ss:parahoricgroup}
We use \emph{parahoric} group schemes over $X$, since their corresponding Hecke stacks and Beilinson--Drinfeld affine Grassmannians in \S\ref{s:grass} enjoy nice properness properties. Let us recall their definition. Let $k$ be a finite field of cardinality $q$, and let $X$ be a connected smooth proper curve over $k$. Write $Q$ for the function field of $X$, fix an algebraic closure $\ov{Q}$ of $Q$, and write $\Ga_Q\coloneqq\Gal(\ov{Q}/Q)$ for the absolute Galois group of $Q$ with respect to $\ov{Q}$. For any closed point $x$ of $X$, write $\cO_x$ for the completion of the local ring $\cO_{X,x}$, and write $Q_x$ for its fraction field.
\begin{defn*}
We call a smooth affine group scheme $G$ over $X$ \emph{parahoric} if it has geometrically connected fibers, its generic fiber $G_Q$ is reductive, and for every closed point $x$ of $X$, the group scheme $G_{\cO_x}$ over $\cO_x$ is parahoric in the sense of \cite[5.2.6]{BT84}.
\end{defn*}
Let $G$ be a parahoric group scheme over $X$. Then there exists a nonempty open subscheme $U$ of $X$ such that $G_U$ is reductive over $U$ \cite[Exp. XIX 2.6]{DG70}. Let $\wt{Q}$ be a finite Galois extension of $Q$ such that the $*$-action on a based root datum of $G_Q$ factors through $\Gal(\wt{Q}/Q)$, and let $\wh{Q}$ be a finite separable extension of $\wt{Q}$ such that $G_{\wh{Q}}$ is split. Write $f:\wh{X}\ra X$ for the finite generically \'etale morphism corresponding to $\wh{Q}/Q$, where $\wh{X}$ is a connected smooth proper curve over $k$. Write $\wh{U}$ for the inverse image $f^{-1}(U)$. After shrinking $U$, we may assume that $G_{\wh{U}}$ is split and $f|_{\wh{U}}$ is \'etale.

Let $T$ be a maximal subtorus of $G_Q$, and let $B$ be a Borel subgroup of $G_{\wh{Q}}$ containing $T_{\wh{Q}}$. After shrinking $U$, we may assume that $T$ extends to a split subtorus of $G_{\wh{U}}$ over $\wh{U}$ and $B$ extends to a Borel subgroup of $G_{\wh{U}}$ over $\wh{U}$. Let $\ell$ be a prime not dividing $q$, and write $(\wh{G},\wh{T},\wh{B})$ for the based dual group over $\ov\bQ_\ell$ associated with the based root datum of $(G_{\wh{Q}},T_{\wh{Q}},B)$. Write $\prescript{L}{}{G}$ for the semidirect product $\wh{G}(\ov\bQ_\ell)\rtimes\Gal(\wt{Q}/Q)$. 

\begin{rem}
Any connected reductive group $G_Q$ over $Q$ arises as the generic fiber of a parahoric group scheme as follows. By spreading out $G_Q$ to a smooth affine group scheme over some nonempty open subscheme $U$ of $X$, applying \cite[Exp. XIX 2.6]{DG70} and \cite[Proposition 3.1.12]{Con14}, and shrinking $U$ if necessary, we obtain a reductive group scheme $G_U$ over $U$ with geometrically connected fibers whose generic fiber is isomorphic to $G_Q$. For the finitely many $x$ in $X\ssm U$, there exists a parahoric group scheme $G_{\cO_x}$ over $\cO_x$ whose generic fiber is isomorphic to $G_{Q_x}$ \cite[5.1.9]{BT84}. Gluing the $G_{\cO_x}$ with $G_U$ via fpqc descent yields a parahoric group scheme over $X$ whose generic fiber is isomorphic to $G_Q$.
\end{rem}

\subsection{}\label{ss:bunG}
We now introduce a general, relative variant of the \emph{moduli space of $G$-bundles on $X$} with level structure. Let $T$ be a scheme over $k$, and let $D$ be a $T$-relative effective Cartier divisor of $X\times T$.
\begin{defn*}
Write $\Bun_{G,D}$ for the prestack over $T$ whose $S$-points parametrize data consisting of
\begin{enumerate}[i)]
\item a $G\times S$-bundle $\cG$ on $X\times S$,
\item an isomorphism $\psi:\cG|_D\ra^\sim (G\times S)|_D$ of $(G\times S)|_D$-bundles.
\end{enumerate}
When $T=k$ and $D=\varnothing$, we shorten this to $\Bun_G$. For $T$-relative effective Cartier divisors $D_1$ and $D_2$ of $X\times T$ such that $D_1\subseteq D_2$, pulling back $\psi$ yields a morphism $\Bun_{G,D_2}\ra\Bun_{G,D_1}$.
\end{defn*}
Now $\Bun_G$ is a smooth algebraic stack over $k$ \cite[Proposition 1]{Hei10}, and note that $\Bun_{G,\varnothing}=\Bun_G\times T$. In general, the Weil restriction $\R_{D/T}((G\times T)|_D)$ has a left action on $\Bun_{G,D}$ via composition with $\psi$, and we see that this exhibits the morphism $\Bun_{G,D}\ra\Bun_G\times T$ as an $\R_{D/T}((G\times T)|_D)$-bundle. Since $\R_{D/T}((G\times T)|_D)$ is a smooth affine group scheme over $T$, we see that $\Bun_{G,D}$ is a smooth algebraic stack over $T$.

\subsection{}\label{ss:div}
In this subsection, we relax our properness assumption on $X$ to separatedness. Let us establish notation on the space of divisors of $X$. Let $d$ be a non-negative integer, and write $\Div_X^d$ for the presheaf over $k$ whose $S$-points parametrize $S$-relative effective Cartier divisors of $X\times S$ with degree $d$. Also, write $X^{(d)}$ for the scheme-theoretic quotient of $X^d$ by the permutation action of the symmetric group $\fS_d$. Since $X$ is a smooth curve over $k$, the morphism $\al:X^d\ra\Div^d_X$ that sends $(x_h)_{h=1}^d\mapsto\sum_{h=1}^d\Ga_{x_h}$ induces an isomorphism $X^{(d)}\ra^\sim\Div_X^d$, where $\Ga_{x_h}$ denotes the graph of $x_h$ \cite[Exp. XVII 6.3.9]{AGV73}.

Write $\Div_X^{d,\circ}$ for the subpresheaf of $\Div_X^d$ whose $S$-points parametrize $S$-relative effective Cartier divisors of $X\times S$ that are \'etale over $S$. We see that the preimage $\al^{-1}(\Div_X^{d,\circ})$ consists of the complement of all diagonals in $X$, so $\Div_X^{d,\circ}$ is an open subscheme of $\Div_X^d$.

\subsection{}\label{ss:divgroups}
In \S\ref{s:grass}, we will apply the relative variant of Definition \ref{ss:bunG} to the following setup. Let $I$ be a finite set. The summation morphism $(\Div_X^d)^I\ra\Div_X^{d\#I}$ corresponds to a $(\Div_X^d)^I$-relative effective Cartier divisor of $X\times(\Div_X^d)^I$ with degree $d\#I$, which we denote by $\Ga_{\sum_{i\in I}D_i}$. For any non-negative integer $n$, write $\Ga_{\sum_{i\in I}nD_i}$ for the $(\Div_X^d)^I$-relative effective Cartier divisor $n\Ga_{\sum_{i\in I}D_i}$ of $X\times(\Div_X^d)^I$, and write $G_{\Ga_{\sum_{i\in I}nD_i}}$ for the Weil restriction
\begin{align*}
\R_{\Ga_{\sum_{i\in I}nD_i}/(\Div_X^d)^I}(G\times_X\Ga_{\sum_{i\in I}nD_i}).
\end{align*}
Note that $G_{\sum_{i\in I}nD_i}$ is a smooth affine group scheme over $(\Div_X^d)^I$. For any $n_1\leq n_2$, we can pull back the counit of the base change-Weil restriction adjunction
\begin{align*}
&G_{\sum_{i\in I}n_2D_i}\times_{(\Div_X^d)^I}\Ga_{\sum_{i\in I}n_2D_i} \\
= \,&\R_{\Ga_{\sum_{i\in I}n_2D_i}/(\Div_X^d)^I}(G\times_X\Ga_{\sum_{i\in I}n_2D_i})\times_{(\Div_X^d)^I}\Ga_{\sum n_2D_i}\ra G\times_X\Ga_{\sum_{i\in I}n_2D_i}
\end{align*}
along $\Ga_{\sum_{i\in I}n_1D_i}\ra\Ga_{\sum_{i\in I}n_2D_i}$ to obtain a morphism
\begin{align*}
G_{\sum_{i\in I}n_2D_i}\times_{(\Div_X^d)^I}\Ga_{\sum_{i\in I}n_1D_i}\ra G\times_X\Ga_{\sum_{i\in I}n_1D_i},
\end{align*}
which induces a morphism $G_{\sum_{i\in I}n_2D_i}\ra G_{\sum_{i\in I}n_1D_i}$ by adjunction.

Write $G_{\sum_{i\in I}\infty D_i}$ for the resulting inverse limit $\varprojlim_nG_{\sum_{i\in I}nD_i}$, which is an affine group scheme over $(\Div_X^d)^I$.

\subsection{}\label{ss:weilrestriction}
We conclude by introducing our Weil restrictions. Let $m:Y\ra X$ be a finite generically \'etale morphism, where $Y$ is a connected smooth proper curve over $k$. Write $F$ for the function field of $Y$, and let $H$ be a parahoric group scheme over $Y$. Applying the discussion in \ref{ss:parahoricgroup} to $H$ over $Y$ yields an open subscheme $V$ of $Y$, a finite Galois extension $\wt{F}$ of $F$, a finite separable extension $\wh{F}$ of $\wt{F}$, a maximal subtorus $A$ of $H_F$, and a Borel subgroup $C$ of $H_{\wh{F}}$. After shrinking $V$, we may assume that $m^{-1}(m(V))=V$ and $m|_V$ is \'etale. Write $U$ for $m(V)$.

Form the Weil restriction $\R_{Y/X}H$. Its generic fiber is the connected reductive group $\R_{F/Q}(H_F)$ over $Q$, and for all closed points $x$ of $X$, we have
\begin{align*}
(\R_{Y/X}H)_{\cO_x} = \R_{(Y\times_X\cO_x)/\cO_x}(H_{Y\times_X\cO_x}) = \prod_{y\in m^{-1}(x)}\R_{\cO_y/\cO_x}(H_{\cO_y}).
\end{align*}
Now \cite[Fact F.1]{Kal19} shows that this is parahoric in the sense of \cite[5.2.6]{BT84}. Thus we may take our parahoric group scheme $G$ to be $\R_{Y/X}H$ in this subsection.

The restriction $G_U$ equals $\R_{V/U}(H_V)$, and because $H_V$ is reductive over $V$ and $m|_V$ is finite \'etale, we see that $G_U$ is reductive over $U$. As the $*$-action of $\Ga_Q$ on a based root datum of $G_Q$ is induced from the $*$-action of $\Ga_F$ on a based root datum of $H_F$, after enlarging $\wt{F}$ we may choose $\wt{Q}=\wt{F}$. Then we may take $\wh{Q}=\wh{F}$. Furthermore, we may choose $T=\R_{F/Q}A$. The natural commutative square 
\begin{align*}
\xymatrix{G_{\wh{F}}\ar[r]^-\sim & \prod_\io H_{\wh{F}} \\
T_{\wh{F}}\ar[r]^-\sim\ar@{^{(}->}[u] & \prod_\io A_{\wh{F}},\ar@{^{(}->}[u]
}
\end{align*}
where $\io$ runs over $\Hom_Q(F,\wh{F})$, indicates that we may take $B=\prod_\io C$. Because $\wh{V}$ is \'etale over $U$, we see that $T_{\wh{F}}$ and $B_{\wh{F}}$ extend over $\wh{V}$.

\subsection{}\label{ss:weilrestrictionbundles}
Maintain the notation of \ref{ss:weilrestriction}, and let $R$ be a scheme over $X$. Note that $\R_{(Y\times_XR)/R}(H\times_XR)=G_R$. Write $\ve:G_{Y\times_X R}\ra H\times_XR$ for the counit of the base change-Weil restriction adjunction, which is a morphism of group schemes over $Y\times_XR$. For any $H\times_XR$-bundle $\cH$ on $Y\times_XR$, the Weil restriction $\R_{(Y\times_XR)/R}\cH$ is a $G_R$-bundle on $R$, as Weil restriction commutes with products. For any $G_R$-bundle $\cG$ on $R$, the pullback $Y\times_X\cG$ is a $G_{Y\times_XR}$-bundle on $Y\times_XR$, so we can form the pushforward $H\times_XR$-bundle $\ve_*(Y\times_X\cG)$.

Since $m$ is a finite morphism of connected curves, \cite[lemma 3.3]{Bre19} shows that this yields an equivalence of categories between $G_R$-bundles on $R$ and $H\times_XR$-bundles on $Y\times_XR$. Let $N$ be a finite closed subscheme of $X$, and write $M$ for $m^{-1}(N)$. By applying this to $R=X\times S$ and $R=N\times S$, we get an isomorphism $c:\Bun_{G,N}\ra^\sim\Bun_{H,M}$.

\section{Hecke stacks and Beilinson--Drinfeld affine Grassmannians}\label{s:grass}
In this section, we introduce symmetrized \emph{Hecke stacks} and \emph{Beilinson--Drinfeld affine Grassmannians}. Instead of parameterizing $G$-bundles on $X$, points on $X$, and isomorphisms between these $G$-bundles away from said points, these symmetrized versions more generally parametrize \emph{divisors} on $X$, along with the other data. This divisorial version appears naturally when taking preimages of points under $m:Y\ra X$.

Our symmetrized Hecke stacks and Beilinson--Drinfeld affine Grassmannians enjoy many of the same properties and structures as in the unsymmetrized special case. We start by defining them, including \emph{convolution} versions thereof, which will be invaluable in \S\ref{s:frob}. Next, using the \emph{Beauville--Laszlo theorem}, we study their relation to each other as well as their relative position stratifications. Finally, we recall the \emph{geometric Satake correspondence}, which describes equivariant perverse sheaves on (usual, unsymmetrized) Beilinson--Drinfeld affine Grassmannians in terms of representations of the dual group.

\subsection{}\label{ss:heckestack}
First, we introduce a symmetrized, convolution version of the \emph{Hecke stack}. Let $I_1,\dotsc,I_k$ be an ordered partition of $I$, and let $N$ be a finite closed subscheme of $X$.
\begin{defn*}
Write $\Hck^{(d)(I_1,\dotsc,I_k)}_{G,N,I}$ for the prestack over $k$ whose $S$-points parametrize data consisting of
\begin{enumerate}[i)]
\item for all $i$ in $I$, a point $D_i$ of $\Div_{X\ssm N}^d(S)$,
\item for all $0\leq j\leq k$, an object $(\cG_j,\psi_j)$ of $\Bun_{G,N}(S)$,
\item for all $1\leq j\leq k$, an isomorphism
  \begin{align*}
\phi_j:\cG_{j-1}|_{X\times S\ssm\sum_{i\in I_j}D_i}\ra^\sim\cG_j|_{X\times S\ssm\sum_{i\in I_j}D_i}
  \end{align*}
 such that $\psi_j\circ\phi_j|_{N\times S}=\psi_{j-1}$.
\end{enumerate}
When $d=1$, we omit it from our notation, and when $N=\varnothing$, we omit it from our notation. For finite closed subschemes $N_1$ and $N_2$ of $X$ such that $N_1\subseteq N_2$, pulling back the $\psi_j$ yields a morphism $\Hck^{(d)(I_1,\dotsc,I_k)}_{G,N_2,I}\ra\Hck^{(d)(I_1,\dotsc,I_k)}_{G,N_1,I}$.

For any $0\leq j\leq k$, write $p_j:\Hck^{(d)(I_1,\dotsc,I_k)}_{G,N,I}\ra\Bun_{G,N}$ for the morphism sending the above data to $(\cG_j,\psi_j)$. We also have a morphism
\begin{align*}
\fp:\Hck^{(d)(I_1,\dotsc,I_k)}_{G,N,I}\ra(\Div_{X\ssm N}^d)^I
\end{align*}
that sends the above data to $(D_i)_{i\in I}$. And if $I_1,\dotsc,I_k$ refines another ordered partition $I'_1,\dotsc,I_{k'}'$ of $I$, we get a morphism
\begin{align*}
\pi^{(I_1,\dotsc,I_k)}_{(I_1',\dotsc,I_{k'}')}:\Hck^{(d)(I_1,\dotsc,I_k)}_{G,N,I}\ra\Hck^{(d)(I_1',\dotsc,I_{k'}')}_{G,N,I}
\end{align*}
by preserving i), preserving $(\cG_0,\psi_0)$, and for all $1\leq j'\leq k'$, taking $\phi_{j'}$ to be the composition of $\phi_j$ over $1\leq j\leq k$ with $I_j\subseteq I_{j'}'$.
\end{defn*}
Because the $D_i$ are disjoint from $N\times S$, for any $0\leq j\leq k$ we see that the commutative square
\begin{align*}
\xymatrix{\Hck^{(d)(I_1,\dotsc,I_k)}_{G,N,I}\ar[r]\ar[d]^-{(p_j,\fp)} & \Hck^{(d)(I_1,\dotsc,I_k)}_{G,I}\ar[d]^-{(p_j,\fp)} \\
\Bun_{G,N}\times(\Div_{X\ssm N}^d)^I \ar[r] & \Bun_G\times(\Div_X^d)^I
}
\end{align*}
is Cartesian. Therefore \ref{ss:bunG} shows that $\Hck^{(d)(I_1,\dotsc,I_k)}_{G,N,I}\ra\Hck^{(d)(I_1,\dotsc,I_k)}_{G,I}|_{(\Div_{X\ssm N}^d)^I}$ is an $\R_{N/k}(G_N)$-bundle. As the morphism 
\begin{align*}
(p_k,\fp):\Hck^{(d)(I_1,\dotsc,I_k)}_{G,I}\ra\Bun_G\times(\Div_X^d)^I
\end{align*}
 is ind-projective \cite[Proposition 3.12]{AH13}\footnote{In \cite{AH13}, only the $d=1$ case is considered. However, the proof of the key step \cite[Proposition 3.7]{AH13} is phrased entirely in terms of relative effective Cartier divisors, so it works for any $d$. Also, \cite{AH13} uses $p_0$ instead of $p_k$, but this makes no difference.}, so we see that $\Hck^{(d)(I_1,\dotsc,I_k)}_{G,I}$ and hence more generally $\Hck^{(d)(I_1,\dotsc,I_k)}_{G,N,I}$ is an ind-algebraic stack over $k$.

\subsection{}\label{defn:BDaffinegrassmannian}
We define similar versions of the \emph{Beilinson--Drinfeld affine Grassmannian}.
\begin{defn*}
Write $\Gr^{(d)(I_1,\dotsc,I_k)}_{G,I}$ for the presheaf over $k$ whose $S$-points parametrize data consisting of
\begin{enumerate}[i)]
\item an object $((D_i)_{i\in I},(\cG_j)_{j=0}^k,(\phi_j)_{j=1}^k)$ of $\Hck^{(d)(I_1,\dotsc,I_k)}_{G,I}(S)$,
\item an isomorphism $\te:\cG_k\ra^\sim G\times S$ of $G\times S$-bundles.
\end{enumerate}
When $d=1$, we omit it from our notation. We have a morphism
\begin{align*}
\fp:\Gr^{(d)(I_1,\dotsc,I_k)}_{G,I}\ra(\Div_X^d)^I
\end{align*}
 as in \ref{ss:heckestack}. If $I_1,\dotsc,I_k$ refines another ordered partition $I_1',\dotsc,I_{k'}'$ of $I$, we also get a morphism $\pi^{(I_1,\dotsc,I_k)}_{(I_1',\dotsc,I_{k'}')}:\Gr^{(d)(I_1,\dotsc,I_k)}_{G,I}\ra\Gr^{(d)(I_1',\dotsc,I_{k'}')}_{G,I}$ as in \ref{ss:heckestack}.
\end{defn*}
Since $\Gr^{(d)(I_1,\dotsc,I_k)}_{G,I}$ is defined via a Cartesian square
\begin{align*}
\xymatrix{\Gr^{(d)(I_1,\dotsc,I_k)}_{G,I}\ar[r]\ar[d] & \Hck^{(d)(I_1,\dotsc,I_k)}_{G,I}\ar[d]^-{p_k}\\
\Spec{k}\ar[r]^-G & \Bun_G,
}
\end{align*}
we see from \ref{ss:heckestack} that $\fp:\Gr^{(d)(I_1,\dotsc,I_k)}_{G,I}\ra(\Div_X^d)^I$ is ind-projective.

\subsection{}\label{ss:unfoldsymmetricaction}
Our symmetrized objects are related to the unsymmetrized special case as follows. Write $[d]$ for the finite set $\{1,\dotsc,d\}$, and for any finite set $J$, write $d\times J$ for $[d]\times J$. We see that the squares
\begin{align*}
\xymatrix{
\Hck^{(d)(I_1,\dotsc,I_k)}_{G,N,I}\ar[d]^-\fp & \Hck^{(d\times I_1,\dotsc,d\times I_k)}_{G,N,d\times I}\ar[l]_-\al\ar[d]^-\fp \\
(\Div_{X\ssm N}^d)^I & (X\ssm N)^{d\times I}\ar[l]_-\al
}
\xymatrix{
\Gr^{(d)(I_1,\dotsc,I_k)}_{G,I}\ar[d]^-\fp & \Gr^{(d\times I_1,\dotsc,d\times I_k)}_{G,d\times I}\ar[l]_-\al\ar[d]^-\fp \\
(\Div_X^d)^I & X^{d\times I}\ar[l]_-\al 
}
\end{align*}
are Cartesian, where the $\al$ send $(x_{h,i})_{h\in[d],i\in I}$ to $(\sum_{h=1}^d\Ga_{x_{h,i}})_{i\in I}$ and preserve all other data. Since the bottom arrows are finite surjective, we see that the top arrows are finite surjective as well. In addition, if $I_1,\dotsc,I_k$ refines another ordered partition $I'_1,\dotsc,I'_{k'}$ of $I$, we see that the squares
\begin{align*}
\xymatrix{
\Hck^{(d)(I_1,\dotsc,I_k)}_{G,N,I}\ar[d]^-{\pi^{(I_1,\dotsc,I_k)}_{(I'_1,\dotsc,I'_{k'})}} & \Hck^{(d\times I_1,\dotsc,d\times I_k)}_{G,N,d\times I}\ar[l]_-\al\ar[d]^-{\pi^{(I_1,\dotsc,I_k)}_{(I'_1,\dotsc,I'_{k'})}} \\
\Hck^{(d)(I_1',\dotsc,I'_{k'})}_{G,N,I} & \Hck^{(d\times I_1',\dotsc,d\times I'_{k'})}_{G,N,d\times I}\ar[l]_-\al 
}
\xymatrix{
\Gr^{(d)(I_1,\dotsc,I_k)}_{G,I}\ar[d]^-{\pi^{(I_1,\dotsc,I_k)}_{(I'_1,\dotsc,I'_{k'})}} & \Gr^{(d\times I_1,\dotsc,d\times I_k)}_{G,d\times I}\ar[l]_-\al\ar[d]^-{\pi^{(I_1,\dotsc,I_k)}_{(I'_1,\dotsc,I'_{k'})}} \\
\Gr^{(d)(I_1',\dotsc,I'_{k'})}_{G,I} & \Gr^{(d\times I_1',\dotsc,d\times I'_{k'})}_{G,d\times I}\ar[l]_-\al 
}
\end{align*}
are also Cartesian.

In all the above squares, note that $\fS_d^I$ has a right action on the right-hand sides via permuting the $(x_{h,i})_{h\in[d],i\in I}$. With respect to this action, the $\al$ are invariant and the right arrows are equivariant.

\subsection{}\label{ss:beauvillelaszlo}
We now recall the \emph{Beauville--Laszlo theorem}. Let $S$ be a scheme over $k$, and let $D$ be an $S$-relative effective Cartier divisor of $X\times S$. For any non-negative integer $n$, the $S$-relative effective Cartier divisor $nD$ of $X\times S$ is finite flat over $S$, so its structure sheaf $\cO_{nD}$ yields a finite flat $\cO_S$-algebra. For any $n_1\leq n_2$, we obtain a morphism $\cO_{n_2D}\ra\cO_{n_1D}$. Write $\cO_D^\wedge$ for the resulting inverse limit $\varprojlim_n\cO_{nD}$, and write $(X\times S)^\wedge_D$ for its relative spectrum $\ul{\Spec}_S\,\cO_D^\wedge$. The contravariance of $\ul{\Spec}_S$ provides a closed immersion $nD\ra(X\times S)^\wedge_D$. By working locally and reducing to affines, we obtain a natural morphism $i:(X\times S)^\wedge_D\ra X\times S$ that preserves the closed subschemes $nD$ \cite[Proposition 2.12.6]{BD99}.

Write $\Vect(X\times S)$ for the category of vector bundles on $X\times S$. Observe that we have an exact tensor functor
\begin{align*}
\Vect(X\times S) &\ra
                                                         \left\{\!\!\!\begin{tabular}{c|l}
                                                           \multirow{3}{*}{$(\cV_1,\cV_2,\vp)$} & $\cV_1$ is a vector bundle on $X\times S\ssm D$,\\
& $\cV_2$ is a vector bundle on $(X\times S)^\wedge_D$, and\\
& $\te:\cV_1|_{(X\times S)^\wedge_D\ssm D}\ra^\sim\cV_2|_{(X\times S)^\wedge_D\ssm D}$
                                                         \end{tabular}\!\right\}
\end{align*}
given by $\cV\mapsto(\cV|_{X\times S\ssm D},\cV|_{(X\times S)^\wedge_D},\id)$.
\begin{thm*}[{\cite[Theorem 2.12.1]{BD99}}]
This yields an equivalence of categories.
\end{thm*}
More generally, the Tannakian description of $G$-bundles \cite[Theorem 4.8]{Bro13} implies that an analogous equivalence of categories holds if we replace ``vector bundle'' everywhere with ``$G$-bundle.''

\subsection{}\label{ss:grfactorization}
Using the Beauville--Laszlo theorem, we get the following reinterpretation of the Beilinson--Drinfeld affine Grassmannian. By pulling back, we see that an $S$-point of $\Gr^{(d)(I_1,\dotsc,I_k)}_{G,I}$ yields data consisting of
\begin{enumerate}[i)]
\item for all $i$ in $I$, a point $D_i$ of $\Div_X^d(S)$,
\item for all $0\leq j\leq k$, a $G|_{(X\times S)^\wedge_{\sum_{i\in I}D_i}}$-bundle $\cG_j$ on $(X\times S)^\wedge_{\sum_{i\in I}D_i}$,
\item for all $1\leq j\leq k$, an isomorphism
  \begin{align*}
\phi_j:\cG_{j-1}|_{(X\times S)^\wedge_{\sum_{i\in I}D_i}\ssm\sum_{i\in I_j}D_i}\ra^\sim\cG_j|_{(X\times S)^\wedge_{\sum_{i\in I}D_i}\ssm\sum_{i\in I_j}D_i},
  \end{align*}
\item an isomorphism $\te:\cG_k\ra^\sim G|_{(X\times S)^\wedge_{\sum_{i\in I}D_i}}$ of $G|_{(X\times S)^\wedge_{\sum_{i\in I}D_i}}$-bundles.
\end{enumerate}
The Beauville--Laszlo theorem enables us to use iii) and iv) to glue ii) with the trivial bundle on $X\times S\ssm\sum_{i\in I}D_i$. Hence conversely $\Gr_{G,I}^{(d)(I_1,\dotsc,I_k)}(S)$ parametrizes precisely the above data.

Write $(\Div_X^d)^I_\circ\subseteq(\Div_X^d)^I$ for the subsheaf of $(D_i)_{i\in I}$ such that the $D_i$ are pairwise disjoint. As the preimage of $(\Div_X^d)^I_\circ$ in $X^{d\times I}$ consists of the complement of certain diagonals, we see that $(\Div_X^d)^I_\circ$ is an open subscheme of $(\Div_X^d)^I$. The above description of $\Gr^{(d)(I_1,\dotsc,I_k)}_{G,I}$ indicates that we have a natural isomorphism
\begin{align*}
\Gr^{(d)(I_1,\dotsc,I_k)}_{G,I}|_{(\Div_X^d)^I_\circ}\ra^\sim\Big(\prod_{i\in I}\Gr^{(d)(i)}_{G,i}\Big)\Big|_{(\Div_X^d)^I_\circ}.
\end{align*}

\subsection{}
The above enables us to decompose the Beilinson--Drinfeld affine Grassmannian according to our ordered partition $I_1,\dotsc,I_k$ as follows. Recall the affine group scheme $G_{\sum_{i\in I}\infty D_i}$ over $(\Div_X^d)^I$ from \ref{ss:divgroups}. The description of $\Gr^{(d)(I_1,\dotsc,I_k)}_{G,I}$ given in \ref{ss:grfactorization} shows that it has a left action of $G_{\sum_{i\in I}\infty D_i}$ via composition with $\te$. This description further indicates that $S$-points of the stack-theoretic quotient $\Gr^{(d)(I_1,\dotsc,I_k)}_{G,I}/G_{\sum_{i\in I}\infty D_i}$ parametrize data consisting of 
\begin{enumerate}[i)]
\item for all $i$ in $I$, a point $D_i$ of $\Div_X^d(S)$,
\item for all $0\leq j\leq k$, a $G|_{(X\times S)^\wedge_{\sum_{i\in I}D_i}}$-bundle $\cG_j$ on $(X\times S)^\wedge_{\sum_{i\in I}D_i}$,
\item for all $1\leq j\leq k$, an isomorphism
  \begin{align*}
    \phi_j:\cG_{j-1}|_{(X\times S)^\wedge_{\sum_{i\in I}D_i}\ssm\sum_{i\in I_j}D_i}\ra^\sim\cG_j|_{(X\times S)^\wedge_{\sum_{i\in I}D_i}\ssm\sum_{i\in I_j}D_i}.
  \end{align*}
\end{enumerate}
In particular, we have a morphism
\begin{align*}
\ka:\Gr^{(d)(I_1,\dotsc,I_k)}_{G,I}&/G_{\sum_{i\in I}\infty D_i}\\
&\ra(\Gr^{(d)(I_1)}_{G,I_1}/G_{\sum_{i\in I_1}\infty D_i})\times\dotsb\times(\Gr^{(d)(I_k)}_{G,I_k}/G_{\sum_{i\in I_k}\infty D_i})
\end{align*}
that sends the above to $(((D_i)_{i\in I_1},(\cG_j)_{j=0}^1,\phi_1),\dotsc,((D_i)_{i\in I_k},(\cG_j)_{j=k-1}^k,\phi_k))$.

\subsection{}
We now explain how the Hecke stack combines the moduli space of $G$-bundles with the Beilinson--Drinfeld affine Grassmannian. Let $n$ be a non-negative integer. Applying Definition \ref{ss:bunG} to $T=(\Div_X^d)^I$ and $D=\Ga_{\sum_{i\in I}nD_i}$ yields a smooth algebraic stack $\Bun_{G,\Ga_{\sum_{i\in I}nD_i}}$ over $(\Div_X^d)^I$. As noted in \ref{ss:bunG}, it is a $G_{\sum_{i\in I}nD_i}$-bundle over $\Bun_G\times(\Div_X^d)^I$, and the $G_{\sum_{i\in I}nD_i}$-action is even defined over $(\Div_X^d)^I$. Write $\Bun_{G,\Ga_{\sum_{i\in I}\infty D_i}}$ for the limit $\varprojlim_n\Bun_{G,\Ga_{\sum_{i\in I}nD_i}}$, which consequently inherits a left action of $G_{\sum_{i\in I}\infty D_i}$.

Consider the stack-theoretic quotient
\begin{align*}
(\Gr^{(d)(I_1,\dotsc,I_k)}_{G,I}\times_{(\Div^d_X)^I}\Bun_{G,\Ga_{\sum_{i\in I}\infty D_i}})/G_{\sum_{i\in I}\infty D_i},
\end{align*}
and write $\cA$ for the prestack over $k$ whose $S$-points parametrize data consisting of
\begin{enumerate}[i)]
\item an object $((D_i)_{i\in I},(\cG_j)_{j=0}^k,(\phi_j)_{j=1}^k)$ of $\Hck^{(d)(I_1,\dotsc,I_k)}_{G,I}(S)$, 
\item an isomorphism $\te:\cG_k|_{(X\times S)^\wedge_{\sum_{i\in I}D_i}}\ra^\sim G|_{(X\times S)^\wedge_{\sum_{i\in I}D_i}}$ of $G|_{(X\times S)^\wedge_{\sum_{i\in I}D_i}}$-bundles.
\end{enumerate}
Note that $G_{\sum_{i\in I}\infty D_i}$ has a left action on $\cA$ via composition with $\te$. We see that this exhibits the natural morphism $\cA\ra\Hck^{(d)(I_1,\dotsc,I_k)}_{G,I}$ as a $G_{\sum_{i\in I}\infty D_i}$-bundle. We also have a morphism
\begin{align*}
\cA\ra\Gr^{(d)(I_1,\dotsc,I_k)}_{G,I}\times_{(\Div^d_X)^I}\Bun_{G,\Ga_{\sum_{i\in I}\infty D_i}}
\end{align*}
given by pulling back $((D_i)_{i\in I},(\cG_j)_{j=0}^k,(\phi_j)_{j=1}^k)$, considering $\cG_k$ in $\Bun_G(S)$, and taking $\te$ for the trivialization. The Beauville--Laszlo theorem implies that this is a $G_{\sum_{i\in I}\infty D_i}$-equivariant isomorphism.

Therefore quotienting by $G_{\sum_{i\in I}\infty D_i}$ induces an isomorphism
\begin{align*}
\Hck^{(d)(I_1,\dotsc,I_k)}_{G,I}\ra^\sim(\Gr^{(d)(I_1,\dotsc,I_k)}_{G,I}\times_{(\Div^d_X)^I}\Bun_{G,\Ga_{\sum_{i\in I}\infty D_i}})/G_{\sum_{i\in I}\infty D_i}.
\end{align*}
Under this identification, write $\de:\Hck^{(d)(I_1,\dotsc,I_k)}_{G,I}\ra\Gr^{(d)(I_1,\dotsc,I_k)}_{G,I}/G_{\sum_{i\in I}\infty D_i}$ for projection onto the first factor.

\subsection{}\label{ss:affinegrassmannian}
We turn to the fibers of the Beilinson--Drinfeld affine Grassmannian. Let $x$ be a closed point of $X$, and write $*$ for the singleton set. The description of $\Gr^{(*)}_{G,*}$ given in \ref{ss:grfactorization} shows that $\Gr^{(*)}_{G,*}|_x$ is naturally isomorphic to the affine Grassmannian of $G_{\cO_x}$ over $\ka(x)$ in the sense of \cite[(1.2.1)]{Zhu17}. Recall that this equals the fpqc sheaf quotient $L(G_{\cO_x})/L^+(G_{\cO_x})$, where $L(G_{\cO_x})$ denotes the loop group of $G_{\cO_x}$ over $\ka(x)$, and $L^+(G_{\cO_x})$ denotes the positive loop group of $G_{\cO_x}$ over $\ka(x)$ \cite[Proposition 1.3.6]{Zhu17}. We see from \ref{defn:BDaffinegrassmannian} that $\Gr^{(*)}_{G,*}|_x$ is an ind-projective scheme over $\ka(x)$.

\subsection{}\label{ss:coweightbounds}
Now we describe the relative position stratification on unsymmetrized affine Grassmannians. Write $X^+_\bullet(T)$ for the set of dominant coweights of $G$ with respect to $T$ and $B$, and let $x$ be a closed point of $U$. Because $G_{\cO_x}$ is reductive, we see that $G_{Q_x}$ is quasi-split and splits over an unramified extension of $Q_x$. 

Let $\om$ be in $X^+_\bullet(T)$, viewed as a dominant coweight of $G_{Q_x}$. Writing $\ka(x)_\om$ for the residue field of the field of definition of $\om$, we see that $\om$ yields a closed affine Schubert variety $\Gr_{x,\om}'\subseteq\Gr^{(*)}_{G,*}|_x\times_x\Spec{\ka(x)_\om}$ as in \cite[p.~83]{Zhu17}. The union of the $\Gal(\ka(x)_\om/\ka(x))$-translates of $\Gr_{x,\om}'$ descends to a closed subvariety $\Gr_{x,\om}\subseteq\Gr^{(*)}_{G,*}|_x$. Recall that $\Gr_{x,\om}'$ and hence $\Gr_{x,\om}$ is projective \cite[Proposition 2.1.5 (1)]{Zhu17}.

Write $\Gr^{(*)}_{G,*,\om}|_U\subseteq\Gr^{(*)}_{G,*}|_U$ for the scheme-theoretic closure of $\bigcup_x\Gr_{x,\om}$ in $\Gr^{(*)}_{G,*}|_U$, where $x$ runs over closed points of $U$. More generally, for $\ul\om=(\om_i)_{i\in I}$ in $X_\bullet^+(T)^I$, write $\Gr^{(I_1,\dotsc,I_k)}_{G,I,\ul\om}|_{U^I}\subseteq\Gr^{(I_1,\dotsc,I_k)}_{G,I}|_{U^I}$ for the scheme-theoretic closure of 
\begin{align*}
\Big(\prod_{i\in I}\Gr^{(i)}_{G,i,\om_i}|_U\Big)\Big|_{U^I_\circ}\subseteq\Gr^{(I_1,\dotsc,I_k)}_{G,I}|_{U^I_\circ}
\end{align*}
in $\Gr^{(I_1,\dotsc,I_k)}_{G,I}|_{U^I}$, where we use \ref{ss:grfactorization} to view the left-hand side as a closed ind-subscheme of the right-hand side. From the projectivity of the $\Gr_{x,\om}$ and the globalization procedure of \cite[Remark 4.3]{Ric14}, we see that $\Gr^{(I_1,\dotsc,I_k)}_{G,I,\ul\om}|_{U^I}$ is projective over $U^I$. Note that $\Gr^{(I_1,\dotsc,I_k)}_{G,I,\ul\om}|_{U^I}$ depends only on the $\Ga_Q^I$-orbit of $\ul\om$.

\subsection{}\label{ss:symmetrizedbounds}
By bootstrapping from \ref{ss:coweightbounds}, we define relative position stratifications on symmetrized Beilinson--Drinfeld affine Grassmannians as follows. View elements of $\fS_d^I$ as bijections $d\times I\ra^\sim d\times I$ that preserve the $I$-factor. Let $\Om$ be a finite $\fS_d^I$-stable and $\Ga_Q^{d\times I}$-stable subset of $X_\bullet^+(T)^{d\times I}$, and write $\Gr^{(d\times I_1,\dotsc,d\times I_k)}_{G,d\times I,\Om}|_{U^{d\times I}}$ for the union 
\begin{align*}
\Gr^{(d\times I_1,\dotsc,d\times I_k)}_{G,d\times I,\Om}|_{U^{d\times I}}\coloneqq\bigcup_{\ul\om\in\Om}\Gr^{(d\times I_1,\dotsc,d\times I_k)}_{G,d\times I,\ul\om}|_{U^{d\times I}}\subseteq\Gr^{(d\times I_1,\dotsc,d\times I_k)}_{G,d\times I}|_{U^{d\times I}}.
\end{align*}
Note that $\Gr^{(d\times I_1,\dotsc,d\times I_k)}_{G,d\times I,\Om}|_{U^{d\times I}}$ is projective over $U^{d\times I}$. As $\Om$ is stable under $\fS_d^I$, we see that $\Gr^{(d\times I_1,\dotsc,d\times I_k)}_{G,d\times I,\Om}|_{U^{d\times I}}$ is also stable under $\fS_d^I$. Therefore, writing $\Gr^{(d)(I_1,\dotsc,I_k)}_{G,I,\Om}|_{(\Div_U^d)^I}$ for the scheme-theoretic image of $\Gr^{(d\times I_1,\dotsc,d\times I_k)}_{G,d\times I,\Om}|_{U^{d\times I}}$ under the morphism
\begin{align*}
\al:\Gr^{(d\times I_1,\dotsc,d\times I_k)}_{G,d\times I}|_{U^{d\times I}}\ra\Gr^{(d)(I_1,\dotsc,I_k)}_{G,I}|_{(\Div^d_U)^I}
\end{align*}
obtained from \ref{ss:unfoldsymmetricaction} via restriction, we see that $\Gr^{(d)(I_1,\dotsc,I_k)}_{G,I,\Om}|_{(\Div_U^d)^I}$ is schematic and proper over $(\Div_U^d)^I$. Moreover, the closed subset of $\Gr^{(d\times I_1,\dotsc,d\times I_k)}_{G,d\times I}|_{U^{d\times I}}$ underlying $\al^{-1}(\Gr^{(d)(I_1,\dotsc,I_k)}_{G,I,\Om}|_{(\Div^d_U)^I})$ is precisely $\Gr^{(d\times I_1,\dotsc,d\times I_k)}_{G,d\times I,\Om}|_{U^{d\times I}}$. If $I_1,\dotsc,I_k$ refines another ordered partition $I'_1,\dotsc,I'_{k'}$ of $I$, we see that $\pi^{(I_1,\dotsc,I_k)}_{(I_1',\dotsc,I'_{k'})}$ sends $\Gr^{(d)(I_1,\dotsc,I_k)}_{G,I,\Om}|_{(\Div^d_U)^I}$ to $\Gr^{(d)(I'_1,\dotsc,I'_{k'})}_{G,I,\Om}|_{(\Div^d_U)^I}$.

\subsection{}\label{ss:repbounds}
It will be useful to index relative position bounds with representations. Write $X^\bullet_+(\wh{T})$ for the set of dominant weights of $\wh{G}$ with respect to $\wh{T}$ and $\wh{B}$. Recall that $\Rep_{\ov\bQ_\ell}(\wh{G}^I)$ is semisimple, and every irreducible object of $\Rep_{\ov\bQ_\ell}(\wh{G}^I)$ can be uniquely written as $\mathlarger{\mathlarger{\boxtimes}}_{i\in I}W_i$, where the $W_i$ are irreducible objects of $\Rep_{\ov\bQ_\ell}\wh{G}$. Now $W_i$ is isomorphic to the Weyl module of a uniquely determined $\om_i$ in $X^\bullet_+(\wh{T})=X_\bullet^+(T)$, so altogether we see that isomorphism classes of objects in $\Rep_{\ov\bQ_\ell}(\wh{G}^I)$ correspond to finite multisets of elements in $X_\bullet^+(T)^I$.

For any $W$ in $\Rep_{\ov\bQ_\ell}((\prescript{L}{}{G})^{d\times I})$, write $\Om(W)$ for the finite subset of $X_\bullet^+(T)^{d\times I}$ underlying the multiset corresponding to $W|_{\wh{G}}$. Then the $\Gal(\wt{Q}/Q)$-action of $\prescript{L}{}{G}$ shows that $\Om$ is $\Ga_Q^{d\times I}$-stable. If $\Om(W)$ is also $\fS_d^I$-stable, write
\begin{align*}
\Gr^{(d)(I_1,\dotsc,I_k)}_{G,I,W}|_{\Div^d_U}\coloneqq\Gr^{(d)(I_1,\dotsc,I_k)}_{G,I,\Om(W)}|_{\Div_U^d}.
\end{align*}
Observe that $\Om(W)$ is always stable under $\fS_d^I$ in the $d=1$ setting.

\subsection{}\label{ss:geometricsatake}
Finally, we recall the \emph{geometric Satake correspondence}. Let $\ze:I\ra J$ be a map of finite sets, and suppose $J_1,\dotsc,J_k$ is an ordered partition of $J$ such that $I_j=\ze^{-1}(J_j)$ for all $1\leq j\leq k$. Now $\ze$ induces morphisms $\ze^*:(\prescript{L}{}{G})^J\ra(\prescript{L}{}{G})^I$ and $\De_\ze:U^J\ra U^I$. We also write $\De_\ze$ for its base change $\Gr_{G,I}^{(I_1,\dotsc,I_k)}|_{U^I}\times_{U^I}U^J\ra\Gr_{G,I}^{(I_1,\dotsc,I_k)}|_{U^I}$. Observe that we may identify $\Gr_{G,I}^{(I_1,\dotsc,I_k)}|_{U^I}\times_{U^I}U^J$ with $\Gr_{G,J}^{(J_1,\dotsc,J_k)}|_{U^J}$.

Write $\sP_{G,I}^{(I_1,\dotsc,I_k)}$ for the category of $G_{\sum_{i\in I}\infty D_i}$-equivariant perverse $\ov\bQ_\ell$-sheaves on $\Gr_{G,I}^{(I_1,\dotsc,I_k)}|_{U^I}$ in the sense of \cite[Sect. A.2]{Gai01}, with degree shifts normalized relative to $U^I$.
\begin{thm*}[{\cite[Theorem 12.16]{Laf16}}\footnote{In \cite{Laf16}, the field $\wt{Q}$ is taken such that $\Gal(\wt{Q}/Q)$ equals the image of $\Ga_Q$ under the $*$-action, and they consider coefficients in a finite extension of $\bQ_\ell$. However, everything works for larger $\wt{Q}$ as well, and extending coefficients to $\ov\bQ_\ell$ is harmless.}] We have a functor
  \begin{align*}
  \Rep_{\ov\bQ_\ell}((\prescript{L}{}G)^I)\ra \sP^{(I_1,\dotsc,I_k)}_{G,I}\mbox{ denoted by }W\mapsto\sS^{(I_1,\dotsc,I_k)}_{I,W}.
  \end{align*}
This functor is fully faithful, and for all $W$ in $\Rep_{\ov\bQ_\ell}((\prescript{L}{}{G})^I)$, it satisfies the following properties:
  \begin{enumerate}[a)]
  \item The perverse sheaf $\sS^{(I_1,\dotsc,I_k)}_{I,W}$ is supported on $\Gr^{(I_1,\dotsc,I_k)}_{G,I,W}|_{U^I}$.
  \item If $I_1,\dotsc,I_k$ refines another ordered partition $I_1',\dotsc,I_{k'}'$ of $I$, we get a natural isomorphism
    \begin{align*}
      (R\pi_{(I_1',\dotsc,I_{k'}')}^{(I_1,\dotsc,I_k)})_!(\sS^{(I_1,\dotsc,I_k)}_{I,W})\ra^\sim\sS^{(I_1',\dotsc,I_{k'}')}_{I,W}.
    \end{align*}
  \item If $W=W_1\boxtimes\dotsb\boxtimes W_k$, where the $W_j$ are objects in $\Rep_{\ov\bQ_\ell}((\prescript{L}{}{G})^{I_j})$, we have a natural isomorphism
    \begin{align*}
      \sS^{(I_1,\dotsc,I_k)}_{I,W}\ra^\sim\ka^*(\sS^{(I_1)}_{I_1,W_1}\boxtimes\dotsb\boxtimes\sS^{(I_k)}_{I_k,W_k}).
    \end{align*}
\item We have a natural isomorphism
  \begin{align*}
    \De_\ze^*(\sS^{(I_1,\dotsc,I_k)}_{I,W})\ra^\sim\sS^{(J_1,\dotsc,J_k)}_{J,W\circ\ze^*}.
  \end{align*}

\item We naturally recover $W$ as the graded derived pushforward
  \begin{align*}
    \bigoplus_{p\in\bZ}(R^p\fp_!\sS^{(I_1,\dotsc,I_k)}_{I,W})(\textstyle\frac{p}2),
  \end{align*}
where $(\frac{p}2)$ denotes the half-integral Tate twist given by our choice of $q^{1/2}$.
  \end{enumerate}
\end{thm*}

\subsection{}\label{ss:ICsheafreps}
We conclude by explicitly describing the functor from Theorem \ref{ss:geometricsatake} in certain cases. Let $\om$ be in $X_\bullet^+(T)$. Write $W_\om$ for the Weyl module corresponding to $\om$, and write $W_{\Ga_Q\cdot\om}$ for the direct sum $\bigoplus_{\om'}W_{\om'}$, where $\om'$ runs over the $\Ga_Q$-orbit of $\om$. Because the $*$-action of $\Ga_Q$ preserves the based dual group $(\wh{G},\wh{T},\wh{B})$, we see that it naturally endows $W_{\Ga_Q\cdot\om}$ with the structure of a finite-dimensional algebraic representation of $\prescript{L}{}G$ over $\ov\bQ_\ell$. Note that $W_{\Ga_Q\cdot\om}$ depends only on the $\Ga_Q$-orbit of $\om$.

The globalization procedure of \cite[p.~139]{Zhu17} and Theorem \ref{ss:geometricsatake}.e) show that the complex $\sS^{(*)}_{*,W_{\Ga_Q\cdot\om}}$ equals the intersection complex of $\Gr^{(*)}_{G,*,\om}|_U$, with degree shifts normalized relative to $U$. More generally, for $\ul\om$ in $X_\bullet^+(T)^I$, write $W_{\Ga_Q^I\cdot\ul\om}$ for the exterior tensor product $\mathlarger{\mathlarger{\boxtimes}}_{i\in I}W_{\Ga_Q\cdot\om_i}$. We see from Theorem \ref{ss:geometricsatake}.c) that $\sS^{(I_1,\dotsc,I_k)}_{I,W_{\Ga_Q^I\cdot\ul\om}}$ equals the intersection complex of $\Gr^{(I_1,\dotsc,I_k)}_{G,I,\ul\om}|_{U^I}$, with degree shifts normalized relative to $U^I$.

\section{Moduli spaces of shtukas}\label{s:shtukas}
Essentially all of \S\ref{s:bung} and \S\ref{s:grass} holds for any perfect field $k$. In contrast, we have Frobenius morphisms when working over a finite field $k$, and in this section we use these Frobenius morphisms to define symmetrized \emph{moduli spaces of shtukas}. These are equi-characteristic analogues of Shimura varieties and their integral models. However, moduli spaces of shtukas admit richer variants than their number field counterparts: namely, the ability to have multiple \emph{legs}, indexed by the finite set $I$. In the unsymmetrized special case, this phenomenon already plays a crucial role in applications to the Langlands program \cite{Laf16, YZ17}, and it also plays a crucial role in this paper.

We start by defining our symmetrized moduli spaces of shtukas and explaining how they inherit various structures from \S\ref{s:bung} and \S\ref{s:grass}. In the usual, unsymmetrized case, we describe how geometric Satake provides coefficient sheaves on the moduli spaces of shtukas, and we recall Xue's result \cite{Xue20b} that their relative cohomology with compact support over $(U\ssm N)^I$ is \emph{ind-smooth}. Finally, we describe \emph{Hecke correspondences} for our symmetrized moduli spaces of shtukas.

\subsection{}
We begin with notation for relative position stratifications on Hecke stacks. For any finite $\fS_d^I$-stable and $\Ga_Q^{d\times I}$-stable subset of $X_\bullet^+(T)^{d\times I}$, write
\begin{align*}
\Hck^{(d)(I_1,\dotsc,I_k)}_{G,I,\Om}|_{(\Div^d_U)^I}\coloneqq\de^{-1}(\Gr^{(d)(I_1,\dotsc,I_k)}_{G,I,\Om}|_{(\Div^d_U)^I}/G_{\sum_{i\in I}\infty D_i}),
\end{align*}
and write
\begin{align*}
\Hck^{(d)(I_1,\dotsc,I_k)}_{G,N,I,\Om}|_{(\Div^d_{U\ssm N})^I}\subseteq\Hck^{(d)(I_1,\dotsc,I_k)}_{G,N,I}|_{(\Div^d_{U\ssm N})^I}
\end{align*}
for the preimage of $\Hck^{(d)(I_1,\dotsc,I_k)}_{G,I,\Om}|_{(\Div^d_{U\ssm N})^I}$. Note that $\Hck^{(d)(I_1,\dotsc,I_k)}_{G,N,I,\Om}|_{(\Div^d_{U\ssm N})^I}$ is a closed substack of $\Hck^{(d)(I_1,\dotsc,I_k)}_{G,N,I}|_{(\Div^d_{U\ssm N})^I}$.

Because $\Gr^{(d)(I_1,\dotsc,I_k)}_{G,I,\Om}|_{(\Div^d_U)^I}$ is schematic and proper over $(\Div_U^d)^I$, we see that $\Hck^{(d)(I_1,\dotsc,I_k)}_{G,N,I,\Om}|_{(\Div^d_{U\ssm N})^I}$ is schematic and proper over $(\Div^d_{U\ssm N})^I\times\Bun_{G,N}$. For any $W$ in $\Rep_{\ov\bQ_\ell}((\prescript{L}{}{G})^{d\times I})$ with $\Om(W)$ stable under $\fS_d^I$, write
\begin{align*}
\Hck^{(d)(I_1,\dotsc,I_k)}_{G,N,I,W}|_{(\Div^d_{U\ssm N})^I}\coloneqq\Hck^{(d)(I_1,\dotsc,I_k)}_{G,N,I,\Om(W)}|_{(\Div^d_{U\ssm N})^I}.
\end{align*}

\subsection{}\label{ss:shtukadefinition}
We have the following symmetrized version of the \emph{moduli space of shtukas}. For any prestack $\cX$ over $k$, write $\Frob_\cX$ or $\Frob$ for its absolute $q$-Frobenius endomorphism.
\begin{defn*}
Write $\Sht^{(d)(I_1,\dotsc,I_k)}_{G,N,I}$ for the stack over $k$ defined by the Cartesian square
\begin{align*}
\xymatrix{\Sht^{(d)(I_1,\dotsc,I_k)}_{G,N,I}\ar[r]\ar[d] & \Hck^{(d)(I_1,\dotsc,I_k)}_{G,N,I}\ar[d]^-{(p_0,p_k)}\\
\Bun_{G,N}\ar[r]^-{(\id,\Frob)} & \Bun_{G,N}\times\Bun_{G,N}
}
\end{align*}
When $d=1$, we omit it from our notation, and when $N=\varnothing$, we omit it from our notation. For finite closed subschemes $N_1$ and $N_2$ of $X$ such that $N_1\subseteq N_2$, we get a morphism $\Sht^{(d)(I_1,\dotsc,I_k)}_{G,N_2,I}\ra\Sht^{(d)(I_1,\dotsc,I_k)}_{G,N_1,I}$ as in \ref{ss:heckestack}. We also have a morphism $\fp:\Sht^{(d)(I_1,\dotsc,I_k)}_{G,N,I}\ra(\Div_{X\ssm N}^d)^I$ as in \ref{ss:heckestack}. And if $I_1,\dotsc,I_k$ refines another ordered partition $I_1',\dotsc,I_{k'}'$ of $I$, we get a morphism $\pi^{(I_1,\dotsc,I_k)}_{(I_1',\dotsc,I_{k'}')}:\Sht^{(d)(I_1,\dotsc,I_k)}_{G,N,I}\ra\Sht^{(d)(I_1',\dotsc,I_{k'}')}_{G,N,I}$ as in \ref{ss:heckestack}.
\end{defn*}
If we replace $\Hck^{(d)(I_1,\dotsc,I_k)}_{G,N,I}$ in the above square with
\begin{align*}
\Hck^{(d)(I_1,\dotsc,I_k)}_{G,N,I,\Om}|_{(\Div^d_{U\ssm N})^I}\mbox{ or }\Hck^{(d)(I_1,\dotsc,I_k)}_{G,N,I,W}|_{(\Div^d_{U\ssm N})^I},
\end{align*}
then we denote the resulting fiber product using
\begin{align*}
\Sht^{(d)(I_1,\dotsc,I_k)}_{G,N,I,\Om}|_{(\Div^d_{U\ssm N})^I}\mbox{ or }\Sht^{(d)(I_1,\dotsc,I_k)}_{G,N,I,W}|_{(\Div^d_{U\ssm N})^I}\mbox{, respectively.}
\end{align*}
We notate $S$-points of $\Sht^{(d)(I_1,\dotsc,I_k)}_{G,N,I}$ using
\begin{align*}
((D_i)_{i\in I},(\cG_0,\psi_0)\dra^{\phi_1}(\cG_1,\psi_1)\dra^{\phi_2}\dotsb\dra^{\phi_{k-1}}(\cG_{k-1},\psi_{k-1})\dra^{\phi_k}(\prescript\tau{}{\cG}_0,\prescript\tau{}{\psi}_0)),
\end{align*}
where $\prescript\tau{}{}$ denotes the pullback $(\id_X\times\Frob_S)^*$. We refer to this as a \emph{shtuka} over $S$, and we call $(D_i)_{i\in I}$ its \emph{legs}.

\subsection{}\label{ss:congruencecovers}
Let us consider level structure covers for moduli spaces of shtukas. Note that $G(N)$ has a left action on $\Sht^{(d)(I_1,\dotsc,I_k)}_{G,N,I}$ via composition with the $\psi_j$. For finite closed subschemes $N_1$ and $N_2$ of $X$ such that $N_1\subseteq N_2$, the morphism $\Sht^{(d)(I_1,\dotsc,I_k)}_{G,N_2,I}\ra\Sht^{(d)(I_1,\dotsc,I_k)}_{G,N_1,I}$ is equivariant with respect to the homomorphism $G(N_2)\ra G(N_1)$.
\begin{prop*}
  This exhibits the morphism
  \begin{align*}
    \Sht^{(d)(I_1,\dotsc,I_k)}_{G,N,I}\ra\Sht^{(d)(I_1,\dotsc,I_k)}_{G,I}|_{(\Div_{X\ssm N}^d)^I}
  \end{align*}
  as a finite Galois morphism with Galois group $G(N)$. In general, this implies that the morphism
\begin{align*}
\Sht^{(d)(I_1,\dotsc,I_k)}_{G,N_2,I}\ra\Sht^{(d)(I_1,\dotsc,I_k)}_{G,N_1,I}|_{(\Div_{X\ssm N_2}^d)^I}
\end{align*}
is finite Galois with Galois group $\ker(G(N_2)\ra G(N_1))$.
\end{prop*}
By pulling back, we obtain a similar statement for $\Sht^{(d)(I_1,\dotsc,I_k)}_{G,N,I,W}|_{(\Div_{U\ssm N}^d)^I}$.
\begin{proof}
The equivariance of $\Sht^{(d)(I_1,\dotsc,I_k)}_{G,N_2,I}\ra\Sht^{(d)(I_1,\dotsc,I_k)}_{G,N_1,I}$ shows that the first statement implies the second. As for the first statement, write $\cB$ for the prestack over $k$ whose $S$-points parametrize a $G_N\times S$-bundle $\cG$ on $N\times S$ along with an isomorphism $\phi:\cG\ra^\sim\prescript\tau{}\cG$ of $G_N\times S$-bundles. Since $N$ is finite over $k$, \cite[lemma 3.3]{Bre19} shows that $G_N\times S$-bundles on $N\times S$ are equivalent to $\R_{N/k}(G_N)_S$-bundles on $S$. Because $G_N$ has geometrically connected fibers, we see that $\R_{N/k}(G_N)$ does as well, so applying \cite[Lemma 3.3 b)]{Var04} to the classifying stack $*/\R_{N/k}(G_N)$ shows that $\cB$ is naturally isomorphic to the discrete stack $(*/\R_{N/k}(G_N))(k)$, which is $*/G(N)$ by Lang's lemma.

Consider the morphism $\Sht^{(d)(I_1,\dotsc,I_k)}_{G,I}|_{(\Div_{X\ssm N}^d)^I}\ra\cB$ given by
\begin{align*}
((D_i)_{i\in I},\cG_0\dra^{\phi_1}\cG_1\dra^{\phi_2}\dotsb\dra^{\phi_{k-1}}\cG_{k-1}\dra^{\phi_k}\prescript\tau{}{\cG}_0)\mapsto(\cG_0|_{N\times S},(\phi_k\dotsb\circ\phi_1)|_{N\times S}).
\end{align*}
Because the $D_i$ are disjoint from $N\times S$, we see that the square
\begin{align*}
\xymatrix{\Sht^{(d)(I_1,\dotsc,I_k)}_{G,N,I}\ar[r]\ar[d] & \Sht^{(d)(I_1,\dotsc,I_k)}_{G,I}|_{(\Div_{X\ssm N}^d)^I}\ar[d]\\
\Spec{k}\ar[r]^-{G(N)} & \cB}
\end{align*}
is Cartesian. As the bottom arrow is finite Galois with Galois group $G(N)$, the top arrow is as well.
\end{proof}

\subsection{}\label{ss:shtukaconvolution}
Convolution morphisms between moduli spaces of shtukas inherit the following properties from their Beilinson--Drinfeld affine Grassmannian counterparts. Write $\ga:\Sht^{(d)(I_1,\dotsc,I_k)}_{G,N,I}\ra\Hck^{(d)(I_1,\dotsc,I_k)}_{G,N,I}$ for the projection morphism. If $I_1,\dotsc,I_k$ refines another ordered partition $I'_1,\dotsc,I'_{k'}$ of $I$, we see that the square
\begin{align*}
\xymatrix{\Sht^{(d)(I_1,\dotsc,I_k)}_{G,I,W}|_{(\Div_U^d)^I}\ar[r]^-{\de\circ\ga}\ar[d]^-{\pi^{(I_1,\dotsc,I_k)}_{(I_1',\dotsc,I_{k'}')}} & \Gr^{(d)(I_1,\dotsc,I_k)}_{G,I,W}|_{(\Div_U^d)^I}/G_{\sum_{i\in I}\infty D_i}\ar[d]^-{\pi^{(I_1,\dotsc,I_k)}_{(I_1',\dotsc,I_{k'}')}}\\
\Sht^{(d)(I'_1,\dotsc,I'_{k'})}_{G,I,W}|_{(\Div_U^d)^I}\ar[r]^-{\de\circ\ga} & \Gr^{(d)(I'_1,\dotsc,I'_{k'})}_{G,I,W}|_{(\Div_U^d)^I}/G_{\sum_{i\in I}\infty D_i}
}
\end{align*}
is Cartesian. Now \ref{ss:symmetrizedbounds} shows that the right arrow is schematic and proper, so the left arrow is as well. In general, we have a Cartesian square
\begin{align*}
\xymatrix{\Sht^{(d)(I_1,\dotsc,I_k)}_{G,N,I,W}|_{(\Div_{U\ssm N}^d)^I}\ar[r]\ar[d]^-{\pi^{(I_1,\dotsc,I_k)}_{(I_1',\dotsc,I_{k'}')}} & \Sht^{(d)(I_1,\dotsc,I_k)}_{G,I,W}|_{(\Div_{U\ssm N}^d)^I}\ar[d]^-{\pi^{(I_1,\dotsc,I_k)}_{(I_1',\dotsc,I_{k'}')}}\\
\Sht^{(d)(I'_1,\dotsc,I'_{k'})}_{G,N,I,W}|_{(\Div_{U\ssm N}^d)^I}\ar[r] & \Sht^{(d)(I'_1,\dotsc,I'_{k'})}_{G,I,W}|_{(\Div_{U\ssm N}^d)^I}},
\end{align*}
which implies that the left arrow here is schematic and proper as well.

\subsection{}\label{ss:delignemumford}
Moduli spaces of shtukas have the following basic geometric structure. 
\begin{prop*}
The stack $\Sht^{(d)(I_1,\dotsc,I_k)}_{G,N,I,W}|_{(\Div^{d}_{U\ssm N})^I}$ is a Deligne--Mumford stack locally of finite type over $(\Div^d_{U\ssm N})^I$.
\end{prop*}
\begin{proof}
  Since the $\pi^{(I_1,\dotsc,I_k)}_{(I'_1,\dotsc,I'_{k'})}$ from \ref{ss:shtukaconvolution} are of finite type, it suffices to consider $k=1$. Furthermore, Proposition \ref{ss:congruencecovers} indicates that it suffices to consider $N=\varnothing$.

  Because $\Bun_G$ is algebraic, we see that $(\id,\Frob):\Bun_G\ra \Bun_G\times\Bun_G$ is of finite type. Hence its base change $\Sht^{(d)(I)}_{G,I,W}|_{(\Div^d_U)^I}\ra\Hck^{(d)(I)}_{G,I,W}|_{(\Div^d_U)^I}$ is as well. Now $\Hck^{(d)(I)}_{G,I,W}|_{(\Div^d_U)^I}$ is schematic and proper over $(\Div^d_U)^I\times\Bun_G$, which itself is an algebraic stack locally of finite type over $(\Div^d_U)^I$, so altogether $\Sht^{(d)(I)}_{G,I,W}|_{(\Div^d_U)^I}$ is also an algebraic stack locally of finite type over $(\Div_U^d)^I$.

To see that $\Sht^{(d)(I)}_{G,I,W}|_{(\Div^d_U)^I}$ is Deligne--Mumford, it suffices to check that its relative diagonal morphism over $(\Div_U^d)^I$ is unramified. As $\Sht^{(d)(I)}_{G,I,W}|_{(\Div^d_U)^I}$ is algebraic, this relative diagonal is already of finite type, so we just need to show that it is formally unramified. The latter follows from the argument on \cite[p.~26--27]{AH13}.
\end{proof}

\subsection{}\label{ss:eps}
Now, we describe our coefficient sheaves in the usual, unsymmetrized case. Write $\eps$ for the composite morphism
\begin{align*}
\xymatrix{
\Sht^{(d)(I_1,\dotsc,I_k)}_{G,N,I}\ar[r] & \Sht^{(d)(I_1,\dotsc,I_k)}_{G,I}\ar[r]^-{\de\circ\ga} & \Gr^{(d)(I_1,\dotsc,I_k)}_{G,I}/G_{\sum_{i\in I}\infty D_i}.
}
\end{align*}
In the $d=1$ setting, write $\sF^{(I_1,\dotsc,I_k)}_{N,I,W}$ for the pullback $\eps^*(\sS^{(I_1,\dotsc,I_k)}_{I,W})$, which Theorem \ref{ss:geometricsatake}.a) enables us to view as a constructible complex of $\ov\bQ_\ell$-sheaves on $\Sht^{(I_1,\dotsc,I_k)}_{G,N,I,W}|_{(U\ssm N)^I}$. Applying proper base change to \ref{ss:shtukaconvolution} and Theorem \ref{ss:geometricsatake}.b) shows that $R\fp_!(\sF^{(I_1,\dotsc,I_k)}_{N,I,W})$ is independent up to isomorphism of the ordered partition $I_1,\dotsc,I_k$, so we denote this ind-(constructible complex of $\ov\bQ_\ell$-sheaves) on $(U\ssm N)^I$ by $\sH_{N,I,W}$. For any integer $p$, write $\sH^{p}_{N,I,W}$ for its $p$-th cohomology, which is an ind-constructible $\ov\bQ_\ell$-sheaf on $(U\ssm N)^I$.

\subsection{}\label{ss:cohomologyissmooth}
To state Xue's result, we recall the definition of \emph{ind-smoothness}. Briefly, let $X$ be any normal connected noetherian scheme over $k$. Recall that we say an ind-constructible $\ov\bQ_\ell$-sheaf $\sM$ on $X$ is \emph{ind-smooth} if $\sM$ is isomorphic to a directed colimit of smooth $\ov\bQ_\ell$-sheaves on $X$. This is equivalent to requiring that, for any geometric points $\ov{x}$ and $\ov{y}$ of $X$ and \'etale path $\ov{y}\rightsquigarrow\ov{x}$, the resulting specialization map $\sM_{\ov{x}}\ra\sM_{\ov{y}}$ is an isomorphism \cite[Lemma 1.1.5]{Xue20b}.

\begin{thm*}[{\cite[Theorem 6.0.12]{Xue20b}}\footnote{Now \cite{Laf16} and \cite{Xue20b} consider relative intersection cohomology of $\Sht^{(I_1,\dotsc,I_k)}_{G,N,I,W}|_{(U\ssm N)^I}\!/\Xi$ instead of $\Sht^{(I_1,\dotsc,I_k)}_{G,N,I,W}|_{(U\ssm N)^I}$, where $\Xi$ is a discrete subgroup of $Z(Q)\bs Z(\bA_Q)$ such that $Z(Q)\bs Z(\bA_Q)/\Xi$ is compact, $Z$ is the center of $G_Q$, and the action of $Z(\bA_Q)$ is given via twisting. However, the key ingredients \cite[Lemma 6.0.6]{Xue20b} and \cite[Lemma 6.0.7]{Xue20b} are proven via geometry on $\Sht^{(I_1,\dotsc,I_k)}_{G,N,I,W}|_{(U\ssm N)^I}$, so they and hence \cite[Theorem 6.0.12]{Xue20b} continue to hold for the latter's relative intersection cohomology. Indeed, this version of \cite[Theorem 6.0.12]{Xue20b} is already crucially used in \cite[Theorem 3.2.3]{AGKRRV21}.}] 
Assume that $X$ is geometrically connected over $k$. Then the ind-constructible $\ov\bQ_\ell$-sheaf $\sH^{p}_{N,I,W}$ on $(U\ssm N)^I$ is ind-smooth.
\end{thm*}
In particular, for any geometric points $\ov{x}$ and $\ov{y}$ of $(U\ssm N)^I$ and \'etale path $\ov{y}\rightsquigarrow\ov{x}$, the specialization morphism $\sH_{N,I,W,\ov{x}}\ra\sH_{N,I,W,\ov{y}}$ in the derived category is an isomorphism.
\begin{rem*}
Even without the geometrically connected assumption on $X$, we expect some form of Theorem \ref{ss:cohomologyissmooth} to hold.
\end{rem*}

\subsection{}\label{ss:shtukasymmetricaction}
Our symmetrized objects are related to the unsymmetrized special case as follows. Recall the Cartesian squares from \ref{ss:unfoldsymmetricaction}. By pulling back along $\ga$, we get analogous Cartesian squares
\begin{align*}
\xymatrix{
\Sht^{(d)(I_1,\dotsc,I_k)}_{G,N,I}\ar[d]^-{\pi^{(I_1,\dotsc,I_k)}_{(I'_1,\dotsc,I'_{k'})}} & \Sht^{(d\times I_1,\dotsc,d\times I_k)}_{G,N,d\times I}\ar[l]_-\al\ar[d]^-{\pi^{(I_1,\dotsc,I_k)}_{(I'_1,\dotsc,I'_{k'})}}\\
\Sht^{(d)(I_1',\dotsc,I'_{k'})}_{G,N,I}\ar[d]^-\fp & \Sht^{(d\times I_1',\dotsc,d\times I'_{k'})}_{G,N,d\times I}\ar[l]_-\al\ar[d]^-\fp \\
(\Div_{X\ssm N}^d)^I & (X\ssm N)^{d\times I}\ar[l]_-\al
}
\end{align*}
for any ordered partition $I'_1,\dotsc,I'_{k'}$ of $I$ refined by $I_1,\dotsc,I_k$.

Let $\Om$ be a finite $\fS_d^I$-stable subset of $X_\bullet^+(T)^{d\times I}$. We see from \ref{ss:symmetrizedbounds} that the above diagram restricts to commutative squares
\begin{align*}
\xymatrix{
\Sht^{(d)(I_1,\dotsc,I_k)}_{G,N,I,\Om}|_{(\Div_{U\ssm N}^d)^I}\ar[d]^-{\pi^{(I_1,\dotsc,I_k)}_{(I'_1,\dotsc,I'_{k'})}} & \displaystyle\Sht^{(d\times I_1,\dotsc,d\times I_k)}_{G,N,d\times I,\Om}|_{(U\ssm N)^{d\times I}}\ar[l]_-\al\ar[d]^-{\pi^{(I_1,\dotsc,I_k)}_{(I'_1,\dotsc,I'_{k'})}}\\
\Sht^{(d)(I'_1,\dotsc,I'_{k'})}_{G,N,I,\Om}|_{(\Div_{U\ssm N}^d)^I}\ar[d]^-\fp & \displaystyle\Sht^{(d\times I'_1,\dotsc,d\times I'_{k'})}_{G,N,d\times I,\Om}|_{(U\ssm N)^{d\times I}}\ar[l]_-\al\ar[d]^-\fp \\
(\Div_{U\ssm N}^d)^I & (U\ssm N)^{d\times I}\ar[l]_-\al
}
\end{align*}
that are Cartesian up to universal homeomorphism. As in \ref{ss:unfoldsymmetricaction}, we see that $\fS_d^I$ has a right action on the right-hand sides for which the $\al$ are invariant and the right arrows are equivariant.

\subsection{}\label{ss:adelicaction}
We conclude by describing \emph{Hecke correspondences} in our setup. First, we define the adelic action at infinite level. Write $\bA_Q$ for the adele ring of $Q$, and write $\bO_Q$ for the integral subring of $\bA_Q$. Write $\eta_{(d)I}$ for the inverse limit $\varprojlim_N(\Div_{X\ssm N}^d)^I$,\footnote{When $X$ is not geometrically connected, $(\Div^d_X)^I$ is not integral. Even when $X$ is geometrically connected, $\eta_{(d)I}$ is not the generic point of $(\Div^d_X)^I$ when $d\geq2$ or $\#I\geq2$.} and write $\Sht^{(d)(I_1,\dotsc,I_k)}_{G,\infty,I}$ for the inverse limit
\begin{align*}
\varprojlim_N\Sht^{(d)(I_1,\dotsc,I_k)}_{G,N,I}|_{\eta_{(d)I}},
\end{align*}
where $N$ runs through finite closed subschemes of $X$. Write $\Sht^{(d)(I_1,\dotsc,I_k)}_{G,\infty,I,W}$ for the analogous inverse limit. By Proposition \ref{ss:congruencecovers}, we see that $\Sht^{(d)(I_1,\dotsc,I_k)}_{G,\infty,I}$ is a pro-Galois cover of $\Sht^{(d)(I_1,\dotsc,I_k)}_{G,I}|_{\eta_{(d)I}}$ with Galois group $G(\bO_Q)$.

We extend the left $G(\bO_Q)$-action to a left $G(\bA_Q)$-action as follows. Note that the $S$-points of $\Sht^{(d)(I_1,\dotsc,I_k)}_{G,\infty,I}$ parametrize data consisting of 
\begin{enumerate}[i)]
\item for all $i$ in $I$, a point $D_i$ of $\eta_{(d)I}(S)$, 
\item for all $0\leq j\leq k-1$, a $G\times S$-bundle $\cG_j$ on $X\times S$ and an isomorphism
  \begin{align*}
    \psi_j:\cG_j|_{\coprod_x(X\times S)^\wedge_{x\times S}}\ra^\sim G\times\textstyle\coprod_x(X\times S)^\wedge_{x\times S}
  \end{align*}
 of $G\times\textstyle\coprod_x(X\times S)^\wedge_{x\times S}$-bundles, where $x$ runs over closed points of $X$,
\item for all $1\leq j\leq k$, an isomorphism
  \begin{align*}
    \phi_j:\cG_{j-1}|_{X\times S\ssm\sum_{i\in I_j}D_i}\ra^\sim\cG_j|_{X\times S\ssm\sum_{i\in I_j}D_i}
  \end{align*}
  such that $\psi_j\circ\phi|_{\coprod_x(X\times S)^\wedge_{x\times S}}=\psi_{j-1}$, where we set $(\cG_k,\psi_k)=(\prescript\tau{}\cG_0,\prescript\tau{}\psi_0)$.
\end{enumerate}
Let $g=(g_x)_x$ be an element of $G(\bA_Q)$. For every closed point $x$ of $X$, we get an isomorphism
\begin{align*}
g_x:G\times((X\times S)_{x\times S}^\wedge\ssm x\times S)\ra^\sim G\times((X\times S)_{x\times S}^\wedge\ssm x\times S),
\end{align*}
and for the cofinitely many $x$ such that $g_x$ lies in $G(\cO_x)$, this extends to an isomorphism over $(X\times S)^\wedge_{x\times S}$. For the finitely many other $x$, the Beauville--Laszlo theorem enables us to use $g_x\circ\psi_j|_{(X\times S)^\wedge_{x\times S}\ssm x\times S}$ to glue $\cG_j|_{(X\ssm x)\times S}$ with the trivial bundle on $(X\times S)_{x\times S}^\wedge$.

Apply this to each of the finitely many other $x$. This yields a $G\times S$-bundle $g\cdot\cG_j$ on $X\times S$ along with an isomorphism $g\cdot\psi_j$ as in ii) such that, for every closed point $x$ of $X$, we have a commutative square
\begin{align*}
\xymatrix{
\cG_j|_{(X\times S)^\wedge_{x\times S}\ssm x\times S}\ar[r]^-\sim\ar[d]^-{\psi_j|_{(X\times S)_{x\times S}^\wedge\ssm x\times S}} & (g\cdot\cG_j)|_{(X\times S)^\wedge_{x\times S}\ssm x\times S}\ar[d]^-{(g\cdot\psi_j)|_{(X\times S)_{x\times S}^\wedge\ssm x\times S}}\\
G\times((X\times S)_{x\times S}^\wedge\ssm x\times S)\ar[r]^-{g_x} & G\times((X\times S)_{x\times S}^\wedge\ssm x\times S).
}
\end{align*}
This procedure is compatible with iii), and it yields a $G(\bA_Q)$-action on $\Sht^{(d)(I_1,\dotsc,I_k)}_{G,\infty,I}$. Since the above square extends to isomorphisms over $(X\times S)^\wedge_{x\times S}$ when $g_x$ lies in $G(\cO_x)$, we see that this indeed extends the $G(\bO_Q)$-action. Finally, compatibility with iii) indicates that the $G(\bA_Q)$-action preserves $\Sht^{(d)(I_1,\dotsc,I_k)}_{G,\infty,I,W}$. 

\subsection{}\label{ss:adeliccompatible}
The adelic action satisfies the following compatibilities. If $I_1,\dotsc,I_k$ refines another ordered partition $I_1',\dotsc,I'_{k'}$ of $I$, the morphisms $\pi^{(I_1,\dotsc,I_k)}_{(I'_1,\dotsc,I'_{k'})}$ pass to the inverse limit in \ref{ss:adelicaction} and yield a morphism
\begin{align*}
\pi^{(I_1,\dotsc,I_k)}_{(I'_1,\dotsc,I'_{k'})}:\Sht^{(d)(I_1,\dotsc,I_k)}_{G,\infty,I}\ra\Sht^{(d)(I'_1,\dotsc,I'_{k'})}_{G,\infty,I}.
\end{align*}
This commutes with the $G(\bA_Q)$-action, and we see that it also restricts to a morphism
\begin{align*}
\pi^{(I_1,\dotsc,I_k)}_{(I'_1,\dotsc,I'_{k'})}:\Sht^{(d)(I_1,\dotsc,I_k)}_{G,\infty,I,W}\ra\Sht^{(d)(I'_1,\dotsc,I'_{k'})}_{G,\infty,I,W}.
\end{align*}
Similarly, the morphisms $\al$ pass to the inverse limit in \ref{ss:adelicaction} and yield a morphism
\begin{align*}
\al:\Sht^{(d\times I_1,\dotsc,d\times I_k)}_{G,\infty,I}\ra\Sht^{(d)(I_1,\dotsc,I_k)}_{G,\infty,I}.
\end{align*}
This commutes with the $G(\bA_Q)$-action, and we see that it also restricts to a morphism
\begin{align*}
\al:\Sht^{(d\times I_1,\dotsc,d\times I_k)}_{G,\infty,I,W}\ra\Sht^{(d)(I_1,\dotsc,I_k)}_{G,\infty,I,W}.
\end{align*}
Note that the right $\fS_d^I$-action on $\Sht^{(d\times I_1,\dotsc,d\times I_k)}_{G,\infty,I}$ and hence $\Sht^{(d\times I_1,\dotsc,d\times I_k)}_{G,\infty,I,W}$ commutes with the $G(\bA_Q)$-action.

\subsection{}\label{ss:heckecorrespondence}
From here, we obtain an action of the Hecke algebra by correspondences as follows. Write $K_{G,N}$ for the subgroup $\ker(G(\bO_Q)\ra G(N))$ of $G(\bO_Q)$, and write $\fH_{G,N}$ for the ring of finitely-supported $\ov\bQ_\ell$-valued functions on $K_{G,N}\bs G(\bA_Q)/K_{G,N}$, where multiplication is given by convolution with respect to the Haar measure on $G(\bA_Q)$ for which $K_{G,N}$ has measure $1$. For any closed point $x$ of $X$, write $\fH_{G,x}$ for the analogous ring of finitely-supported $\ov\bQ_\ell$-valued functions on $G(\cO_x)\bs G(Q_x)/G(\cO_x)$. Recall that $\fH_{G,N}$ contains the restricted tensor product $\bigotimes_u'\fH_{G,u}$, where $u$ runs over closed points of $X\ssm N$.

Let $g$ be in $G(\bA_Q)$. The $G(\bA_Q)$-action from \ref{ss:adelicaction} yields a finite \'etale correspondence
\begin{align*}
\xymatrix{\Sht^{(d)(I_1,\dotsc,I_k)}_{G,\infty,I}/(K_{G,N}\cap g^{-1}K_{G,N}g)\ar[r]^-g_-\sim\ar[d] & \Sht^{(d)(I_1,\dotsc,I_k)}_{G,\infty,I}/(gK_{G,N}g^{-1}\cap K_{G,N})\ar[d]\\
\Sht^{(d)(I_1,\dotsc,I_k)}_{G,N,I}|_{\eta_{(d)I}} & \Sht^{(d)(I_1,\dotsc,I_k)}_{G,N,I}|_{\eta_{(d)I}},
}
\end{align*}
using the fact that $\Sht^{(d)(I_1,\dotsc,I_k)}_{G,\infty,I}/K_{G,N} = \Sht^{(d)(I_1,\dotsc,I_k)}_{G,N,I}|_{\eta_{(d)I}}$. By sending the indicator function of $K_{G,N}gK_{G,N}$ to the above correspondence, we obtain a ring homomorphism from $\fH_{G,H}$ to the ring of $\ov\bQ_\ell$-valued finite \'etale correspondences on $\Sht^{(d)(I_1,\dotsc,I_k)}_{G,N,I}|_{\eta_{(d)I}}$ over $\eta_{(d)I}$. We similarly obtain finite \'etale correspondences on $\Sht^{(d)(I_1,\dotsc,I_k)}_{G,N,I,W}|_{\eta_{(d)I}}$.

We see from \ref{ss:adeliccompatible} that our correspondences are compatible with $\pi^{(I_1,\dotsc,I_k)}_{(I'_1,\dotsc,I'_{k'})}$ and $\al$. In the $d=1$ setting, proper base change shows that they induce an action of $\fH_{G,N}$ on $\sH_{N,I,W}|_{\eta_{(d)I}}$.

\begin{rem}\label{rem:extension}
Let $N(g)$ be a finite set of closed points $x$ of $X$ containing those for which $g$ does not lie in $G(\cO_x)$. Note that the construction in \ref{ss:adelicaction} more generally yields a left $\big(\prod_{x\notin N(g)}G(\cO_x)\big)\times\big(\prod_{x\in N(g)}G(Q_x)\big)$-action on
\begin{align*}
\varprojlim_N\Sht^{(d)(I_1,\dotsc,I_k)}_{G,N,I}|_{(\Div^d_{X\ssm N(g)})^I},
\end{align*}
where $N$ runs through finite closed subschemes of $X$ supported on $N(g)$, that extends the left $G(\bO_Q)$-action. Therefore the construction in \ref{ss:heckecorrespondence} naturally extends to a finite \'etale correspondence on $\Sht^{(d)(I_1,\dotsc,I_k)}_{G,N,I}|_{(\Div^d_{X\ssm(N\cup N(g))})^I}$, which restricts to one on $\Sht^{(d)(I_1,\dotsc,I_k)}_{G,N,I,W}|_{(\Div^d_{U\ssm(N\cup N(g))})^I}$. 
\end{rem}

\section{Partial Frobenii and derived categories}\label{s:frob}
In positive characteristic algebraic geometry, forming fundamental groups rarely commutes with taking products---even over an algebraically closed field. One can remedy this by asking for additional structure: namely, \emph{partial Frobenius morphisms}. We crucially need to carry out this strategy in the derived category, which forces us to use $\infty$-categorical structures.

First, we define partial Frobenii and describe their action on derived categories of $\ov\bQ_\ell$-sheaves. Next, we explain a monodromy interpretation of partial Frobenii, which generalizes the relation between Weil sheaves and Weil groups. We then describe how partial Frobenii arise in the setting of symmetrized moduli spaces of shtukas. In the usual, unsymmetrized case, we state an anticipated result of Arinkin--Gaitsgory--Kazhdan--Raskin--Rozenblyum--Varshavsky showing that their relative cohomology complex satisfies a derived version of \emph{Drinfeld's lemma}. Finally, we use work of Xue to relate the action of Frobenius elements in the Weil group with that of partial Frobenii on our cohomology groups.

\subsection{}\label{ss:derivedcategories}
Let us start with notation on derived categories. For any algebraic stack $\cX$ over $\ov{k}$, write $\Shv(\cX)$ for the derived category of $\ov\bQ_\ell$-sheaves on $\cX$ as in \cite[Sect. 1.1.1]{AGKRRV20b}. For any morphism $f:\cX\ra\cY$ of algebraic stacks over $\ov{k}$, we have functors $Rf^!:\Shv(\cY)\ra\Shv(\cX)$ and $Rf_*:\Shv(\cX)\ra\Shv(\cY)$ \cite[Sects. 1.1.1--1.1.2]{AGKRRV20b}, and by using the alternative description of $\Shv(\cX)$ from \cite[(C.1)]{AGKRRV20a}, we similarly obtain functors $f^*:\Shv(\cY)\ra\Shv(\cX)$ and $Rf_!:\Shv(\cX)\ra\Shv(\cY)$. Moreover, these constructions are functorial in $f$.

Now let $\cX_1$ and $\cX_2$ be smooth quasicompact schemes over $\ov{k}$. Then the external tensor product functor
\begin{align*}
\Shv(\cX_1)\otimes\Shv(\cX_2)\ra\Shv(\cX_1\times_{\ov{k}}\cX_2)
\end{align*}
is fully faithful \cite[Lemma A.2.6]{GKRV19}, where $\otimes$ denotes the Lurie tensor product over $\ov\bQ_\ell$ \cite[\S4.8.1]{Lu17}.

\subsection{}\label{ss:partialfrobeniusgeometry}
Products of spaces over $k$ have the following Frobenius structures. Let $I$ be a finite set, and let $\cX$ be a prestack over $k$. For any subset $I_1$ of $I$, write $\Frob_{I_1}:\cX^I\ra\cX^I$ for the product $\big(\prod_{i\in I_1}\Frob_\cX\big)\times\big(\prod_{i\in I\ssm I_1}\id_\cX\big)$. Note that $\Frob_{I_1}$ and $\Frob_{I_2}$ commute for any subsets $I_1$ and $I_2$ of $I$, and for any partition $I_1,\dotsc,I_k$ of $I$, the composition of $\Frob_{I_j}$ for all $1\leq j\leq k$ equals the absolute $q$-Frobenius endomorphism of $\cX$.

Next, let $\ze:I\ra J$ be a map of finite sets, and let $f:\cX\ra\cY$ be a morphism of prestacks over $k$. This induces a morphism $f^\ze:\cX^J\ra\cY^I$, and if $I_1=\ze^{-1}(J_1)$ for some subset $J_1$ of $J$, we get a commutative square
\begin{align*}
  \xymatrix{\cX^J\ar[r]^-{f^\ze}\ar[d]^-{\Frob_{J_1}} & \cY^I\ar[d]^-{\Frob_{I_1}}\\
\cX^J\ar[r]^-{f^\ze} & \cY^I.
  }
\end{align*}

Now suppose $\cX$ is an algebraic stack over $k$. Because $\Frob_{I_1}$ and hence $\Frob_{I_1,\ov{k}}$ are universal homeomorphisms, we see that $\Frob_{I_1,\ov{k}}^*$ yields an equivalence on categories of \'etale sheaves. By applying this to subsets $I_1,\dotsc,I_k$ of $I$ and using the functoriality of $\Shv$, we obtain an action of $\bZ^k$ on $\Shv(\cX^I_{\ov{k}})$. Finally, if $f$ is a morphism of algebraic stacks over $k$, we obtain a functor $f^{\ze,*}_{\ov{k}}:\Shv(\cY^I_{\ov{k}})\ra\Shv(\cX^J_{\ov{k}})$ that is compatible with this action.

\subsection{}\label{ss:fweil}
Next, we turn to a group-theoretic version of \ref{ss:partialfrobeniusgeometry}. Let $\cX$ be a geometrically connected algebraic stack over $k$. Since $\Frob_{I_1,\ov{k}}$ is a universal homeomorphism, it induces an automorphism of $\pi_1(\cX_{\ov{k}})$ as a topological group. By applying this to subsets $I_1,\dotsc,I_k$ of $I$, we obtain a continuous action of $\bZ^k$ on $\pi_1(\cX_{\ov{k}})$. Write $\FWeil_I^{I_1,\dotsc,I_k}(\cX)$ for the resulting semidirect product $\pi_1(\cX_{\ov{k}})\rtimes\bZ^k$. When the $I_j$ are precisely the singletons in $I$, we write $\FWeil_I(\cX)$ instead, noting that this recovers the Weil group $\Weil(\cX)$ when $I$ itself is a singleton.

Let $\ze:I\ra J$ be a map of finite sets, let $f:\cX\ra\cY$ be a morphism of geometrically connected algebraic stacks over $k$, and suppose that $I_j=\ze^{-1}(J_j)$ for some subsets $J_1,\dotsc,J_k$ of $J$. We see from \ref{ss:partialfrobeniusgeometry} that $f^*_{\ze,\ov{k}}:\pi_1(\cY_{\ov{k}}^J)\ra\pi_1(\cX_{\ov{k}}^I)$ commutes with the $\bZ^k$-action. This induces a continuous homomorphism
\begin{align*}
\FWeil_J^{J_1,\dotsc,J_k}(\cY)\ra\FWeil_I^{I_1,\dotsc,I_k}(\cX).
\end{align*}
In particular, by applying this to the inclusion of singletons in $I$ and $f=\id_\cX$, we obtain a continuous homomorphism $\FWeil_I(\cX)\ra\Weil(\cX)$ for each $i$ in $I$. Together, they yield a continuous homomorphism
\begin{align*}
\FWeil_I(\cX)\ra\Weil(\cX)^I,
\end{align*}
which is surjective because $\pi_1(\cX_{\ov{k}}^I)\ra\pi_1(\cX_{\ov{k}})^I$ is surjective.

\subsection{}\label{ss:partialfrobeniusmonodromy}
We now recall the monodromy interpretation of Frobenius--Weil groups, as well as \emph{Drinfeld's lemma}. Let $\cX$ be a geometrically connected algebraic stack over $k$. By restricting to $\pi_1(\cX_{\ov{k}}^I)$ and separately considering the action of $\bZ^k$, we see that finite-dimensional continuous representations of $\FWeil^{I_1,\dotsc,I_k}_I(\cX)$ over $\ov\bQ_\ell$ are equivalent to smooth $\ov\bQ_\ell$-sheaves on $\cX^I_{\ov{k}}$ equipped with commuting $\Frob_{I_j,\ov{k}}$-semilinear automorphisms. The latter are called \emph{partial Frobenius morphisms}. Moreover, a similar equivalence holds for ind-smooth $\ov\bQ_\ell$-sheaves, and precomposing with the continuous homomorphisms from \ref{ss:fweil} corresponds to pullback.

Now suppose $\cX$ is a noetherian scheme. Then, as in the proof of \cite[Theorem 5.6]{HRS20}, we see that \cite[Lemme 8.11]{Laf16} and \cite[Lemma 5.7]{HRS20} imply that any finite-dimensional continuous representation of $\FWeil_I(\cX)$ over $\ov\bQ_\ell$ factors through $\Weil(\cX)^I$ via the continuous homomorphism $\FWeil_I(\cX)\ra\Weil(\cX)^I$ from \ref{ss:fweil}.\footnote{Strictly speaking, \cite[Lemma 5.7]{HRS20} applies over finite extensions of $\bQ_\ell$, not $\ov\bQ_\ell$. But $\FWeil_I(\cX)$ is an extension of a finitely-generated group by a compact group, so its finite-dimensional continuous representations over $\ov\bQ_\ell$ are defined over finite extensions of $\bQ_\ell$.}

\subsection{}\label{ss:partialfrobeniusgaloisfrobenius}
Partial Frobenii and Frobenius elements in the Weil group are related as follows. Maintain the assumptions of \ref{ss:partialfrobeniusmonodromy}, and suppose $\cX$ is geometrically integral over $k$. Let $\ov{\eta_I}$ be a geometric generic point of $\cX^I_{\ov{k}}$. Let $k'$ be a finite extension of $k$ with degree $r$, and let $\ul{x}=(x_i)_{i\in I}$ be a point of $\cX^I(k')$. Choose a geometric point $\ov{\ul{x}}$ of $\cX_{\ov{k}}^I$ lying over $\ul{x}$, as well as an \'etale path $\ov{\eta_I}\rightsquigarrow\ov{\ul{x}}$.

Write $\pr_i:\cX^I\ra\cX$ for projection onto the $i$-th factor. This induces an \'etale path $\pr_i(\ov{\eta_I})\rightsquigarrow\ov{x_i}$, from which we naturally obtain a continuous homomorphism $\Weil(x_i)\ra\Weil(\cX)$, where we form $\Weil(x_i)$ using absolute $q^r$-Frobenius instead of absolute $q$-Frobenius as in \ref{ss:fweil}. Write $\ga_{x_i}$ for the image of the generator of $\bZ=\Weil(x_i)$ in $\Weil(\cX)$.

Let $\sM$ be a smooth $\ov\bQ_\ell$-sheaf on $\cX_{\ov{k}}^I$ equipped with partial Frobenii
\begin{align*}
\ov{F}_{i}:\Frob_{i,\ov{k}}^*\sM\ra^\sim\sM.
\end{align*}
Now we have a specialization isomorphism $\sM|_{\ov{\ul{x}}}\ra^\sim\sM|_{\ov{\eta_I}}$. On the one hand, \ref{ss:partialfrobeniusmonodromy} endows $\sM|_{\ov{\eta_I}}$ with an action of $\Weil(\cX)^I$. On the other hand, we see
\begin{align*}
\ov{F}_{i}\circ\Frob^*_{i,\ov{k}}(\ov{F}_{i})\circ\dotsb\circ\Frob^{r-1,*}_{i,\ov{k}}(\ov{F}_{i})
\end{align*}
restricts to an automorphism of $\sM|_{\ov{\ul{x}}}$.
\begin{prop*}
Under the specialization isomorphism $\sM|_{\ov{\ul{x}}}\ra^\sim\sM|_{\ov{\eta_I}}$, this automorphism corresponds to the action of $\ga_{x_i}$ in the $i$-th entry of $\Weil(\cX)^I$.
\end{prop*}
\begin{proof}
  By using \ref{ss:partialfrobeniusmonodromy} to pass to finite-dimensional continuous representations of $\Weil(\cX)^I$ over $\ov\bQ_\ell$, we obtain a surjection $\mathlarger{\mathlarger{\boxtimes}}_{i\in I}\sE_i\ra\sM$ compatible with partial Frobenii, where the $\sE_i$ are Weil $\ov\bQ_\ell$-sheaves on $\cX$. Thus it suffices to prove the claim for $\mathlarger{\mathlarger{\boxtimes}}_{i\in I}\sE_i$. Now the identification $(\mathlarger{\mathlarger{\boxtimes}}_{i\in I}\sE_i)|_{\ov{\ul{x}}}=\bigotimes_{i\in I}(\cE|_{\ov{x_i}})$ is compatible with partial Frobenii, and $(\mathlarger{\mathlarger{\boxtimes}}_{i\in I}\sE_i)|_{\ov{\eta_I}}=\bigotimes_{i\in I}(\cE|_{\pr_i(\ov{\eta_I})})$ is compatible with the $\Weil(\cX)^I$-action. So it suffices to prove the claim for $\sE_i$ and $I=\{i\}$. But $\bZ=\Weil(x_i)\ra\Weil(\cX)$ sends the generator of $\bZ$ to $r$ times the generator of $\bZ\subseteq\Weil(\cX)$, and by definition this image acts via
  \begin{gather*}
\ov{F}_{i}\circ\Frob^*_{i,\ov{k}}(\ov{F}_{i})\circ\dotsb\circ\Frob^{r-1,*}_{i,\ov{k}}(\ov{F}_{i})\qedhere
  \end{gather*}
\end{proof}

\subsection{}\label{ss:partialfrobeniusshtukas}
We now specialize to the setting of moduli spaces of shtukas. Consider the morphism $\Fr^{(d)(I_1,\dotsc,I_k)}_{I_1,N,I}:\Sht^{(d)(I_1,\dotsc,I_k)}_{G,N,I}\ra\Sht^{(d)(I_2,\dotsc,I_k,I_1)}_{G,N,I}$ that sends
\begin{align*}
\hspace{-.2cm}((D_i)_{i\in I\ssm I_1},(D_i)_{i\in I_1},(\cG_0,\psi_0)\dra^{\phi_1}(\cG_1,\psi_1)\dra^{\phi_2}\dotsb\dra^{\phi_{k-1}}(\cG_{k-1},\psi_{k-1})\dra^{\phi_k}(\prescript\tau{}{\cG}_0,\prescript\tau{}{\psi}_0))
\end{align*}
to
\begin{align*}
\hspace{-.15cm}((D_i)_{i\in I\ssm I_1},(\prescript\tau{}{D}_i)_{i\in I_1},(\cG_1,\psi_1)\dra^{\phi_2}(\cG_2,\psi_2)\dra^{\phi_3}\dotsb\dra^{\phi_k}(\prescript\tau{}{\cG}_0,\prescript\tau{}{\psi}_0)\dra^{\prescript\tau{}{\phi}_1}(\prescript\tau{}{\cG}_1,\prescript\tau{}{\psi}_1)).
\end{align*}
Let $\Om$ be a finite $\fS_d^I$-stable and $\Ga_Q^{d\times I}$-stable subset of $X_\bullet^+(T)^{d\times I}$. We see that $\Fr^{(d)(I_1,\dotsc,I_k)}_{I_1,N,I}$ sends $\Sht^{(d)(I_1,\dotsc,I_k)}_{G,N,I,\Om}|_{(\Div^d_{U\ssm N})^I}$ to $\Sht^{(d)(I_2,\dotsc,I_k,I_1)}_{G,N,I,\Om}|_{(\Div^d_{U\ssm N})^I}$.

Next, observe that the square
\begin{align*}
\xymatrixcolsep{2cm}
\xymatrix{\Sht^{(d)(I_1,\dotsc,I_k)}_{G,N,I}\ar[r]^-{\Fr^{(d)(I_1,\dotsc,I_k)}_{I_1,N,I}}\ar[d]^-\fp & \Sht^{(d)(I_2,\dotsc,I_k,I_1)}_{G,N,I}\ar[d]^-\fp \\
(\Div_{X\ssm N}^d)^I\ar[r]^-{\Frob_{I_1}} & (\Div_{X\ssm N}^d)^I
}
\end{align*}
commutes. Finally, note that the composition
\begin{align*}
\xymatrixcolsep{2cm}
\xymatrix{\Sht^{(d)(I_1,\dotsc,I_k)}_{G,N,I}\ar[r]^-{\Fr^{(d)(I_1,\dotsc,I_k)}_{I_1,N,I}} & \dotsb \ar[r]^-{\Fr^{(d)(I_k,I_1,\dotsc,I_{k-1})}_{I_k,N,I}} & \Sht^{(d)(I_1,\dotsc,I_k)}_{G,N,I}
}
\end{align*}
equals $\Frob:\Sht^{(d)(I_1,\dotsc,I_k)}_{G,N,I}\ra\Sht^{(d)(I_1,\dotsc,I_k)}_{G,N,I}$. Since $\Frob$ is a universal homeomorphism, this shows that $\Fr^{(d)(I_1,\dotsc,I_k)}_{I_1,N,I}$ is a universal homeomorphism too.

\subsection{}\label{ss:partialfrobeniuscompatibility}
Our $\Fr^{(d)(I_1,\dotsc,I_k)}_{I_1,N,I}$ are compatible with other structures as follows. If $I_1,\dotsc,I_k$ refines another ordered partition $I_1',\dotsc,I_{k'}'$ of $I$ such that $I'_1 = I_1\cup\dotsb\cup I_j$, we obtain a commutative diagram
\begin{align*}
\xymatrixcolsep{2cm}
  \xymatrix{\Sht^{(d)(I_1,\dotsc,I_k)}_{G,N,I}\ar[r]^-{\Fr^{(d)(I_1,\dotsc,I_k)}_{I_1,N,I}}\ar[d]^-{\pi^{(I_1,\dotsc,I_k)}_{(I'_1,\dotsc,I'_{k'})}} & \dotsb \ar[r]^-{\Fr^{(d)(I_j,I_1,\dotsc,I_{j-1})}_{I_j,N,I}} & \Sht^{(d)(I_{j+1},\dotsc,I_k,I_1,\dotsc,I_j)}_{G,N,I}\ar[d]^-{\pi^{(I_1,\dotsc,I_k)}_{(I'_1,\dotsc,I'_{k'})}}\\
  \Sht^{(d)(I_1',\dotsc,I_{k'}')}_{G,N,I}\ar[rr]^-{\Fr^{(d)(I'_1,\dotsc,I'_{k'})}_{I'_1,N,I}} & & \Sht^{(d)(I_2',\dotsc,I_{k'}')}_{G,N,I}.
}
\end{align*}
We also have a commutative square
\begin{align*}
  \xymatrix{
  \Sht^{(d)(I_1,\dotsc,I_k)}_{G,N,I}\ar[d]^-{\Fr^{(d)(I_1,\dotsc,I_k)}_{I_1,N,I}} & \ar[l]_-\al\Sht^{(d\times I_1,\dotsc,d\times I_k)}_{G,N,d\times I}\ar[d]^-{\Fr^{(d\times I_1,\dotsc,d\times I_k)}_{d\times I_1,N,d\times I}}\\
  \Sht^{(d)(I_2,\dotsc,I_k,I_1)}_{G,N,I} & \ar[l]_-\al\Sht^{(d\times I_2,\dotsc,d\times I_k,d\times I_1)}_{G,N,d\times I}.
  }
\end{align*}
Finally, note that the $\Fr^{(d)(I_1,\dotsc,I_k)}_{I_1,N,I}$ pass to the inverse limit in \ref{ss:adelicaction}. The resulting morphism $\Fr^{(d)(I_1,\dotsc,I_k)}_{I_1,\infty,I}:\Sht^{(d)(I_1,\dotsc,I_k)}_{G,\infty,I}\ra\Sht^{(d)(I_2,\dotsc,I_k,I_1)}_{G,\infty,I}$ evidently commutes with the $G(\bA_Q)$-action, so we see that $\Fr^{(d)(I_1,\dotsc,I_k)}_{I_1,N,I}$ commutes with the finite \'etale correspondences from \ref{ss:heckecorrespondence} and Remark \ref{rem:extension}.

\subsection{}\label{ss:shtukapf}
In the usual, unsymmetrized case, we convert our $\Fr^{(I_1,\dotsc,I_k)}_{I_1,N,I}$ into partial Frobenius morphisms as follows. Let $W$ be in $\Rep_{\ov\bQ_\ell}((\prescript{L}{}{G})^{I})$. Then \ref{ss:partialfrobeniusshtukas} and the proof of \cite[Proposition 3.3]{Laf16}\footnote{Now \cite[Proposition 3.3]{Laf16} only treats the case of split $G$. However, it extends to the general case, which is already implicitly used in \cite[\S12]{Laf16}.} yield a canonical isomorphism
\begin{align*}
\Fr^{(I_1,\dotsc,I_k),*}_{I_1,N,I}\sF^{(I_2,\dotsc,I_k,I_1)}_{N,I,W}\ra^\sim\sF^{(I_1,\dotsc,I_k)}_{N,I,W}.
\end{align*}
Since $\Fr^{(I_1,\dotsc,I_k)}_{I_1,N,I}$ and $\Frob_{I_1}$ are universal homeomorphisms, applying proper base change to the commutative square from \ref{ss:partialfrobeniusshtukas} and arguing as in \cite[\S4.3]{Laf16} yields an isomorphism
\begin{align*}
  F_{I_1}:\Frob_{I_1}^*\sH_{N,I,W}\ra^\sim\sH_{N,I,W}
\end{align*}
of ind-(constructible complexes of $\ov\bQ_\ell$-sheaves) on $(U\ssm N)^{I}$.

By \ref{ss:partialfrobeniuscompatibility}, we see that $F_{I_1}$ is compatible with the independence of $\sH_{N,I,W}$ on the ordered partition $I_1,\dotsc,I_k$ up to isomorphism. In particular, reordering the $I_1,\dotsc,I_k$ shows that
\begin{align*}
F_{I_2}\circ\Frob_{I_2}^*(F_{I_1}) = F_{I_1}\circ\Frob_{I_1}^*(F_{I_2})
\end{align*}
on $\sH^p_{N,I,W}$. We can also use \ref{ss:partialfrobeniuscompatibility} to see that $F_{I_1}|\eta_{(1)I}$ commutes with the action of $\fH_{G,N}$ on $\sH_{N,I,W}|_{\eta_{(1)I}}$.

\subsection{}\label{ss:AGKRRV}
We now state the anticipated result of Arinkin--Gaitsgory--Kazhdan--Raskin--Rozenblyum--Varshavsky. By applying \ref{ss:partialfrobeniusgeometry}, we see that the fully faithful functor $\Shv((U\ssm N)_{\ov{k}})^{\otimes I}\hookrightarrow\Shv((U\ssm N)_{\ov{k}}^I)$ is equivariant for the action of $\bZ^I$. In particular, we naturally have a fully faithful functor
\begin{align*}
(\Shv((U\ssm N)_{\ov{k}})^{\otimes I})^{B\bZ^I}\hookrightarrow\Shv((U\ssm N)_{\ov{k}}^I)^{B\bZ^I}.
\end{align*}
\begin{thm*}\footnote{While \cite{AGKRRV21} only treats the case of split $G$, their methods adapt to the general case.}
Assume that $X$ is geometrically connected over $k$. Then we have a functor
  \begin{align*}
    \Rep_{\ov\bQ_\ell}((\prescript{L}{}{G})^I) \ra (\Shv((U\ssm N)_{\ov{k}})^{\otimes I})^{B\bZ^I}\mbox{ denoted by }W\mapsto\sT_{N,I,W}
  \end{align*}
  satisfying the following properties:
  \begin{enumerate}[a)]
  \item for all $W$ in $\Rep_{\ov\bQ_\ell}((\prescript{L}{}{G})^I)$, the image of $\sT_{N,I,W}$ in $\Shv((U\ssm N)_{\ov{k}}^I)$ is naturally isomorphic to the pullback $\sH_{N,I,W,\ov{k}}$ of $\sH_{N,I,W}$ to $(U\ssm N)_{\ov{k}}^I$, and the equivariance data corresponds to the pullbacks $\ov{F}_{i}$ of the $F_{i}$,

  \item as $I$ varies, there exist $\infty$-categorical coherences for $\sT_{N,I,W}$ in the sense of \cite[Sect. 1.6]{AGKRRV21}.
  \end{enumerate}
\end{thm*}
\begin{proof}
When $U\ssm N=X$, take $\sT_{N,I,-}$ to be $\Sht^{\text{Tr}}_I(-)$ as in \cite[Sect. 4.1.1]{AGKRRV21}. This has the desired structure by \cite[Sect. 4.5.3]{AGKRRV21}. Part b) follows since the inclusion $\Shv((U\ssm N)_{\ov{k}})^{\otimes I}\hookrightarrow\Shv((U\ssm N)_{\ov{k}}^I)$ is compatible with changing $I$, and part a) follows from \cite[Theorem 4.1.2]{AGKRRV21} and \cite[Sect. 4.5.4]{AGKRRV21}. The general case is forthcoming work of Arinkin--Gaitsgory--Kazhdan--Raskin--Rozenblyum--Varshavsky.
\end{proof}
\begin{rem*}
Even without the geometrically connected assumption on $X$, we expect some form of Theorem \ref{ss:AGKRRV} to hold.
\end{rem*}

\subsection{}\label{ss:symmetrizedescend}
Maintain the assumptions of Theorem \ref{ss:AGKRRV}, which we will use to define symmetrized versions of $\sT_{N,I,W}$ as follows. Note that \ref{ss:partialfrobeniusgeometry} yields an action of $\fS_d^I$ on $\Shv((U\ssm N)_{\ov{k}}^{d\times I})$ intertwining the action of $\bZ^{d\times I}$, and observe that
\begin{align*}
\Shv((U\ssm N)_{\ov{k}})^{\otimes(d\times I)}\hookrightarrow\Shv((U\ssm N)_{\ov{k}}^{d\times I})
\end{align*}
is equivariant for this action, where $\Shv((U\ssm N)_{\ov{k}})^{\otimes(d\times I)}$ has the natural $\fS_d^I$-action.

Let $W$ be in $\Rep_{\ov\bQ_\ell}((\prescript{L}{}{G})^{d\times I})$, and suppose $W\circ\sg^*=W$ for all $\sg$ in $\fS_d^I$. Then Theorem \ref{ss:AGKRRV}.b) gives an $\fS_d^I$-equivariance structure on the object $\sT_{N,d\times I,W}$ of $(\Shv((U\ssm N)_{\ov{k}})^{\otimes(d\times I)})^{B\bZ^{d\times I}}$, so we obtain an object of
\begin{align*}
  ((\Shv((U\ssm N)_{\ov{k}})^{\otimes(d\times I)})^{B\bZ^{d\times I}})^{B\fS_d^I} &= (\Shv((U\ssm N)_{\ov{k}})^{\otimes(d\times I)})^{B(\bZ^{d\times I}\rtimes\fS_d^I)} \\
  & = ((\Shv((U\ssm N)_{\ov{k}})^{B\bZ})^{\otimes(d\times I)})^{B\fS_d^I},
\end{align*}
where we use \cite[Proposition 2.5.7]{GR17a} to see that $\otimes$ commutes with taking equivariant objects. 

By finite \'etale descent, we see that pullback yields a natural equivalence
\begin{align*}
\Shv((U\ssm N)^{d\times I}_{\ov{k}}/\fS_d^I)\ra^\sim\Shv((U\ssm N)^{d\times I}_{\ov{k}})^{B\fS_d^I},
\end{align*}
where $(U\ssm N)^{d\times I}_{\ov{k}}/\fS_d^I=((U\ssm N)_{\ov{k}}^d/\fS_d)^I$ denotes the stack-theoretic quotient. Thus if we only remember the action of $\bZ^I\times\fS_d^I\subseteq\bZ^{d\times I}\rtimes\fS_d^I$, then we obtain an object $\sT^{(d)}_{N,I,W}$ of
\begin{align*}
\Shv((U\ssm N)_{\ov{k}}^{d\times I})^{B(\bZ^I\times\fS_d^I)} = \Shv(((U\ssm N)^d/\fS_d)^I_{\ov{k}})^{B\bZ^I}.
\end{align*}
Finally, since $\fS_d$ preserves and acts freely on the open subscheme $(U\ssm N)^d_\circ$ of $(U\ssm N)^d$, we can identify the stack-theoretic quotient $((U\ssm N)^d_\circ/\fS_d)^I$ with $(\Div^{d,\circ}_{U\ssm N})^I$ via \ref{ss:div}. In particular, restricting $\sT_{N,I,W}^{(d)}$ to $((U\ssm N)^d_{\circ,\ov{k}}/\fS_d)^I$ yields an object of
\begin{align*}
\Shv((\Div^{d,\circ}_{U\ssm N})^I_{\ov{k}})^{B\bZ^I}.
\end{align*}

\subsection{}\label{ss:fiberfunctors}
From here, we obtain complexes of representations as follows. Write $\eta$ for the generic point of $X$, and note that taking fibers at $\eta$ yields a functor
\begin{align*}
\Shv((U\ssm N)_{\ov{k}})^{B\bZ}\ra D(\Weil(\eta),\ov\bQ_\ell).
\end{align*}
By taking $(d\times I)$-th tensor powers and postcomposing with the exterior product, we get a functor
\begin{align*}
\Shv((U\ssm N)_{\ov{k}})^{B\bZ})^{\otimes(d\times I)}\ra D(\Weil(\eta)^{d\times I},\ov\bQ_\ell).
\end{align*}
Finally, taking $\fS_d^I$-equivariant objects gives us a functor
\begin{align*}
((\Shv((U\ssm N)_{\ov{k}})^{B\bZ})^{\otimes(d\times I)})^{B\fS_d^I}\ra D(\Weil(\eta)^{d\times I}\rtimes\fS_d^I,\ov\bQ_\ell).
\end{align*}

\subsection{}\label{ss:galfrob}
We conclude by using work of Xue to bootstrap Proposition \ref{ss:partialfrobeniusgaloisfrobenius} to the ind-smooth $\ov\bQ_\ell$-sheaves $\sH^p_{N,I,W}$ on $(U\ssm N)^I$. Assume that $X$ is geometrically connected over $k$, and apply the notation in \ref{ss:partialfrobeniusgaloisfrobenius} to $\cX=U\ssm N$. By projecting to $(U\ssm N)^I$, we view $\ov{\eta_I}$ and $\ov{\ul{x}}$ as geometric points of $(U\ssm N)^I$.

We have a specialization isomorphism $\sH^p_{N,I,W}|_{\ov{\ul{x}}}\ra^\sim\sH^p_{N,I,W}|_{\ov{\eta_I}}$ by Theorem \ref{ss:cohomologyissmooth}. On the one hand, Theorem \ref{ss:AGKRRV}.a) and Theorem \ref{ss:cohomologyissmooth} endow $\sH^p_{N,I,W}|_{\ov{\eta_I}}$ with an action of $\Weil(\eta)^I$, which one can show factors through $\Weil(U\ssm N)^I$ \cite[Proposition 6.0.13]{Xue20b}. On the other hand, we see that
\begin{align*}
F_{i}\circ\Frob^*_{i}(F_{i})\circ\dotsb\circ\Frob^{r-1,*}_{i}(F_{i})
\end{align*}
restricts to an automorphism of $\sH^p_{N,I,W}|_{\ov{\ul{x}}}$.
\begin{prop*}
Under the identification $\sH^p_{N,I,W}|_{\ov{\ul{x}}}\ra^\sim\sH^p_{N,I,W}|_{\ov{\eta_I}}$, this automorphism corresponds to the action of $\ga_{x_i}$ in the $i$-th entry of $\Weil(U\ssm N)^I$.
\end{prop*}
\begin{proof}
  Now \cite[Lemma 6.0.9]{Xue20b} and Theorem \ref{ss:cohomologyissmooth} show that $\sH^p_{N,I,W}$ is a union of ind-smooth $\ov\bQ_\ell$-subsheaves $\sM$ over $(U\ssm N)^I$, where $\sM|_{\ov{\eta_I}}$ is preserved by and finitely generated over $\bigotimes_{i\in I}\fH_{G,u_i}$ for some closed points $u_i$ of $U\ssm N$, and the $F_{i}$ restrict to isomorphisms on $\sM$. Hence it suffices to prove the analogous claim for $\sM$.

  Let $\fm$ be a maximal ideal of $\bigotimes_{i\in I}\fH_{G,u_i}$. Since the $\fH_{G,u_i}$ are finitely generated $\ov\bQ_\ell$-algebras, we see that $(\sM|_{\ov{\eta_I}})/\fm^n$ corresponds to a smooth $\ov\bQ_\ell$-sheaf on $(U\ssm N)^I$ equipped with partial Frobenii. By base changing to $(U\ssm N)_{\ov{k}}$, Proposition \ref{ss:partialfrobeniusgaloisfrobenius} proves the analogous claim for this smooth $\ov\bQ_\ell$-sheaf.

Because $\bigotimes_{i\in I}\fH_{G,u_i}$ is noetherian, the map $\sM|_{\ov{\eta_I}}\ra{}\textstyle\bigoplus_\fm(\sM|_{\ov{\eta_I}})^\wedge_\fm$ is injective, where $\fm$ runs over maximal ideals of $\bigotimes_{i\in I}\fH_{G,u_i}$, and $(-)^\wedge_\fm$ denotes $\fm$-adic completion. We similarly obtain an injection $\sM|_{\ov{\ul{x}}}\hookrightarrow\textstyle\bigoplus_\fm(\sM|_{\ov{\ul{x}}})^\wedge_\fm$. As the partial Frobenii and $\bigotimes_{i\in I}\fH_{G,u_i}$-actions commute, we see that these injections preserve the relevant structures. Thus the analogous claim for the sheaves corresponding to the $(\sM|_{\ov{\eta_I}})/\fm^n$ implies the claim for $\sM$, as desired.
\end{proof}

\section{The plectic conjecture}\label{s:plectic}
In this section, we prove our results on the plectic conjecture for (usual, unsymmetrized) moduli spaces of shtukas. We begin by using the relationship between Weil restriction and $\Bun_G$ from \S\ref{s:bung} to describe one incarnation of the conjectured \emph{plectic diagram} from \cite[(1.3)]{NS16}. This relates unsymmetrized shtukas for $G$ to symmetrized shtukas for $H$. From here, we use the link between symmetrized and unsymmetrized shtukas for $H$ to prove Theorem A. Next, we use the Hecke compatibility of this relation to prove Theorem B. We conclude by explicating our constructions to prove Theorem C.

\subsection{}\label{ss:plectichecke}
First, we describe the Hecke stack incarnation of the \emph{plectic diagram}. Recall the notation of \ref{ss:weilrestrictionbundles}, and write $d$ for the degree of $m$. Since $Y$ and hence $\Div^d_Y$ is proper, the morphism $m^{-1}:X\ra\Div^d_Y$ sending $x\mapsto m^{-1}(\Ga_x)$ is proper as well. As $m^{-1}$ is also a monomorphism, we see that it is a closed immersion. Because $m$ is \'etale over $U$, we see that $m^{-1}$ sends $U$ to $\Div^{d,\circ}_Y$.

For any $0\leq j_0\leq k$, recall the morphisms $p_{j_0}$ and $\fp$ from Definition \ref{ss:heckestack}.
\begin{prop*}
We have a natural Cartesian square
\begin{align*}
\xymatrixcolsep{3pc}
\xymatrix{\Hck^{(I_1,\dotsc,I_k)}_{G,N,I}\ar[r]\ar[d]^-{(p_{j_0},\fp)} & \Hck^{(d)(I_1,\dotsc,I_k)}_{H,M,I}\ar[d]^-{(p_{j_0},\fp)}\\
\Bun_{G,N}\times(X\ssm N)^I\ar[r]^-{c\times(m^{-1})^I} & \Bun_{H,M}\times(\Div_{Y\ssm M}^d)^I.
}
\end{align*}
\end{prop*}
\begin{proof}
An $S$-point of $\Hck^{(I_1,\dotsc,I_k)}_{G,N,I}$ consists of
\begin{enumerate}[i)]
\item for all $i$ in $I$, a point $x_i$ of $(X\ssm N)(S)$,
\item for all $0\leq j\leq k$, an object $(\cG_j,\psi_j)$ of $\Bun_{G,N}(S)$,
\item for all $1\leq j\leq k$, an isomorphism
  \begin{align*}
   \phi_j:\cG_{j-1}|_{X\times S\ssm\sum_{i\in I_j}\Ga_{x_i}}\ra^\sim\cG_j|_{X\times S\ssm\sum_{i\in I_j}\Ga_{x_i}}
  \end{align*}
 with $\psi_j\circ\phi_j|_{N\times S}=\psi_{j-1}$.
\end{enumerate}
By using the isomorphism $c$ and applying \ref{ss:weilrestrictionbundles} to $R=X\times S\ssm\sum_{i\in I_j}\Ga_{x_i}$, we see that ii) and iii) are equivalent to objects $(\cH_j,\psi'_j)$ of $\Bun_{H,M}(S)$ along with isomorphisms
\begin{align*}
\phi_j':\cH_{j-1}|_{Y\times S\ssm\sum_{i\in I_j}m^{-1}(\Ga_{x_i})}\ra^\sim\cH_j|_{Y\times S\ssm\sum_{i\in I_j}m^{-1}(\Ga_{x_i})}
\end{align*}
satisfying $\psi'_j\circ\phi'_j|_{M\times S}=\psi_{j-1}'$. Combined with i), this is precisely the data parametrized by the fiber product of $\Hck^{(d)(I_1,\dotsc,I_k)}_{H,M,I}$ and $\Bun_{G,N}\times(X\ssm N)^I$ over the product $\Bun_{H,M}\times(\Div^d_{Y\ssm M})^I$.
\end{proof}

\subsection{}\label{ss:plecticgr}
We obtain the Beilinson--Drinfeld affine Grassmannian version of the plectic diagram as follows. Take $N=\varnothing$ and $j_0=k$ in Proposition \ref{ss:plectichecke}. After pulling back along the $k$-point of $\Bun_G\ra^\sim\Bun_H$ corresponding to the trivial bundle, this yields a Cartesian square
\begin{align*}
\xymatrix{\Gr^{(I_1,\dotsc,I_k)}_{G,I}\ar[r]\ar[d]^-{\fp} & \Gr^{(d)(I_1,\dotsc,I_k)}_{H,I}\ar[d]^-{\fp}\\
X^I\ar[r]^-{(m^{-1})^I} & (\Div_{Y}^d)^I.
}  
\end{align*}

\subsection{}\label{ss:plecticgrbounds}
Next, we explain how the relative position stratification fits into the plectic diagram. Because $T=\R_{F/Q}A$, we see that $X_\bullet(T)$ is isomorphic as a $\Ga_Q$-module to 
\begin{align*}
\{\vp:\Ga_Q\ra X_\bullet(A)\mid \vp(xg)=g^{-1}\vp(x)\mbox{ for all }g\in\Ga_F\mbox{ and }x\in \Ga_Q\},
\end{align*}
whose $\Ga_Q$-action is given by inverse left multiplication. Under this identification, $X_\bullet(T)^{\Ga_Q}$ corresponds to the subset of functions $\vp$ taking constant values in $X_\bullet(A)^{\Ga_F}$. As $B=\prod_\io C$, we see that $X_\bullet^+(T)$ corresponds to the subset of functions $\vp$ that take values in $X_\bullet^+(A)$. After choosing representatives for $\Ga_Q/\Ga_F$ in $\Ga_Q$ and enumerating them, we may identify $X_\bullet(T)$ with $X_\bullet(A)^d$ and $X_\bullet^+(T)$ with $X_\bullet^+(A)^d$.

Let $\Om$ be a finite $\fS_d^I$-stable and $\Ga_F^{d\times I}$-stable subset of $X_\bullet^+(A)^{d\times I}$. In particular, $\Om$ is also $\Ga_Q^I$-stable when viewed as a subset of $X_\bullet^+(T)^I$. Hence we can form $\Gr^{(d)(I_1,\dotsc,I_k)}_{H,I,\Om}|_{(\Div^d_V)^I}$ and $\Gr^{(I_1,\dotsc,I_k)}_{G,I,\Om}|_{U^I}$ as in \ref{ss:symmetrizedbounds}. Because $(m^{-1})^I$ is a locally closed immersion, we see that the Cartesian square in \ref{ss:plecticgr} restricts to a Cartesian square
\begin{align*}
\xymatrix{\Gr^{(I_1,\dotsc,I_k)}_{G,I,\Om}|_{U^I}\ar[r]\ar[d]^-\fp & \Gr^{(d)(I_1,\dotsc,I_k)}_{H,I,\Om}|_{(\Div_V^d)^I}\ar[d]^-\fp\\
U^I\ar[r]^-{(m^{-1})^I} & (\Div_V^d)^I.
}
\end{align*}
\begin{rem*}
Note that stability under $\Ga^{d\times I}_F$ is a condition that is independent of our choice of representatives for $\Ga_Q/\Ga_F$ in $\Ga_Q$. Similarly, stability under $\fS_d^I$ is a condition that is independent of our enumeration of said representatives. 
\end{rem*}

\subsection{}\label{ss:plecticgrsheaves}
In the plectic setting, we work with the following relative position strata and corresponding sheaves. First, suppose $\Om$ equals $\prod_{i\in I}\Om_{i}$, where the $\Om_i$ are finite $\fS_d$-stable and $\Ga_F^d$-stable subsets of $X_\bullet^+(A)^d$. In particular, $\Om_{i}$ is a finite disjoint union of $\Ga_F^d$-orbits $O$. Then we can form $W_O$ as in \ref{ss:ICsheafreps}, and we write $W_{\Om,i,H}$ for the object $\bigoplus_OW_O$ of $\Rep_{\ov\bQ_\ell}((\prescript{L}{}{H})^d)$. Note that the $\fS_d$-stability of $\Om_{i}$ implies that $W_{\Om,i,H}\circ\sg^*=W_{\Om,i,H}$ for all $\sg$ in $\fS_d$. Finally, write $W_{\Om,H}$ for the exterior tensor product $\mathlarger{\mathlarger{\boxtimes}}_{i\in I}W_{\Om,i,H}$, and recall from \ref{ss:ICsheafreps} that $\sS^{(d\times I_1,\dotsc,d\times I_k)}_{d\times I,W_{\Om,H}}$ equals the intersection complex of $\Gr_{H,d\times I,\Om}^{(d\times I_1,\dotsc,d\times I_k)}|_{V^{d\times I}}$, with degree shifts normalized relative to $V^{d\times I}$.

By viewing $\Om_i$ as a subset of $X_\bullet^+(T)$ instead, we see that it is $\Ga_Q$-stable. Form $W_{\Om,i,G}$ and $W_{\Om,G}$ as above. We analogously see that $\sS^{(I_1,\dotsc,I_k)}_{I,W_{\Om,G}}$ equals the intersection complex of $\Gr^{(I_1,\dotsc,I_k)}_{G,I,\Om}|_{U^I}$, with degree shifts normalized relative to $U^I$.

More generally, let $\Om$ be any finite subset of $X_\bullet^+(A)^{d\times I}$ that is $\fS_d^I$-stable and $\Ga_F^{d\times I}$-stable. Then $\Om$ is a finite union of subsets of the form considered above, and we write $W_{\Om,H}$ and $W_{\Om,G}$ for the corresponding direct sum of algebraic representations over $\ov\bQ_\ell$. We see that the above relation with intersection complexes continues to hold.

\subsection{}\label{ss:plecticsht}
Finally, we arrive at the plectic diagram for moduli spaces of shtukas. By pulling back \ref{ss:plectichecke} along $\ga$, we get an analogous Cartesian square
\begin{align*}
\xymatrix{\Sht^{(I_1,\dotsc,I_k)}_{G,N,I}\ar[r]\ar[d]^-{\fp} & \Sht^{(d)(I_1,\dotsc,I_k)}_{H,M,I}\ar[d]^-{\fp}\\
(X\ssm N)^I\ar[r]^-{(m^{-1})^I} & (\Div_{Y\ssm M}^d)^I.
}
\end{align*}
Moreover, further restricting to \ref{ss:plecticgrbounds} via $\ga\circ\de$ yields a Cartesian square
\begin{align*}
\xymatrix{\Sht^{(I_1,\dotsc,I_k)}_{G,N,I,\Om}|_{(U\ssm N)^I}\ar[r]\ar[d]^-{\fp} & \Sht^{(d)(I_1,\dotsc,I_k)}_{H,M,I,\Om}|_{(\Div_{V\ssm M}^d)^I}\ar[d]^-{\fp}\\
(U\ssm N)^I\ar[r]^-{(m^{-1})^I} & (\Div_{V\ssm M}^d)^I.
} 
\end{align*}
If $I_1,\dotsc,I_k$ refines another ordered partition $I_1',\dotsc,I_{k'}'$, of $I$, note that
\begin{align*}
\xymatrix{
\Sht^{(I_1,\dotsc,I_k)}_{G,N,I}\ar[r]\ar[d]^-{\pi^{(I_1,\dotsc,I_k)}_{(I_1',\dotsc,I_{k'}')}} & \Sht^{(d)(I_1,\dotsc,I_k)}_{H,M,I}\ar[d]^-{\pi^{(I_1,\dotsc,I_k)}_{(I_1',\dotsc,I_{k'}')}}\\
\Sht^{(I_1',\dotsc,I_{k'}')}_{G,N,I}\ar[r] & \Sht^{(d)(I_1',\dotsc,I_{k'}')}_{H,M,I}
}
\end{align*}
yields a commutative square. Also, we have a commutative square
\begin{align*}
\xymatrix{
\Sht^{(I_1,\dotsc,I_k)}_{G,N,I}\ar[r]\ar[d]^-{\Fr_{I_1,N,I}^{(I_1,\dotsc,I_k)}} & \Sht^{(d)(I_1,\dotsc,I_k)}_{H,M,I}\ar[d]^-{\Fr_{I_1,M,I}^{(d)(I_1,\dotsc,I_k)}}\\
\Sht^{(I_2,\dotsc,I_k,I_1)}_{G,N,I}\ar[r] & \Sht^{(d)(I_2,\dotsc,I_k,I_1)}_{H,M,I}.
}
\end{align*}

\subsection{}\label{ss:ICpullbacks}
We now transfer \ref{ss:plecticgrsheaves} to moduli spaces of shtukas by using the morphism $\eps$ from \ref{ss:eps}. In the $d=1$ setting, the proof of \cite[Proposition 2.11]{Laf16}\footnote{Now \cite[Proposition 2.11]{Laf16} only treats the case of split $G$. However, it extends to the general case, which is already implicitly used in \cite[\S12]{Laf16}.} indicates that $\eps$ \'etale-locally induces an isomorphism from $\Sht_{G,N,I}^{(I_1,\dotsc,I_k)}|_{(U\ssm N)^I}$ to $\Gr^{(I_1,\dotsc,I_k)}_{G,I}|_{(U\ssm N)^I}$. Therefore $\sF^{(I_1,\dotsc,I_k)}_{N,I,W_{\Om,G}}=\eps^*(\sS^{(I_1,\dotsc,I_k)}_{I,W_{\Om,G}})$ equals the intersection complex of $\Sht_{G,N,I,\Om}^{(I_1,\dotsc,I_k)}|_{(U\ssm N)^I}$, with degree shifts normalized relative to $(U\ssm N)^I$. Applying this discussion to $H$ shows that $\sF^{(d\times I_1,\dotsc,d\times I_k)}_{M,d\times I,W_{\Om,H}}$ equals the intersection complex of $\Sht^{(d\times I_1,\dotsc,d\times I_k)}_{H,M,d\times I,\Om}|_{(V\ssm M)^{d\times I}}$, with degree shifts normalized relative to $(V\ssm M)^{d\times I}$.

Write $\sF^{(d)(I_1,\dotsc,I_k)}_{M,I,W_{\Om,H}}$ for the intersection complex of $\Sht^{(d)(I_1,\dotsc,I_k)}_{H,M,I,\Om}|_{(\Div^d_{V\ssm M})^I}$, with degree shifts normalized relative to $(\Div_{V\ssm M}^d)^I$. In the commutative squares from \ref{ss:plecticsht}, note that the horizontal arrows are locally closed immersions, up to universal homeomorphism. Hence we can identify the pullback of $\sF^{(d)(I_1,\dotsc,I_k)}_{M,I,W_{\Om,H}}$ to $\Sht^{(I_1,\dotsc,I_k)}_{G,N,I,\Om}|_{U^I}$ with $\sF^{(I_1,\dotsc,I_k)}_{N,I,W_{\Om,G}}$. And because the commutative squares from \ref{ss:shtukasymmetricaction} are finite generically \'etale, we can similarly identify the pullback of $\sF^{(d)(I_1,\dotsc,I_k)}_{M,I,W_{\Om,H}}$ to $\Sht^{(d\times I_1,\dotsc,d\times I_k)}_{H,M,d\times I,\Om}|_{(V\ssm M)^{d\times I}}$ with $\sF^{(d\times I_1,\dotsc,d\times I_k)}_{M,d\times I,W_{\Om,H}}$.

\subsection{}\label{ss:physicaldescent}
We leverage the link between symmetrized and unsymmetrized shtukas for $H$ to get the following relation on cohomology. Restricting the commutative square from \ref{ss:shtukasymmetricaction} to $(\Div_{V\ssm M}^{d,\circ})^I$ yields a Cartesian square
\begin{align*}
\xymatrix{\Sht^{(d)(I_1,\dotsc,I_k)}_{H,M,I,\Om}|_{(\Div_{V\ssm M}^{d,\circ})^I}\ar[d]^-\fp & \ar[l]_-\al\Sht^{(d\times I_1,\dotsc,d\times I_k)}_{H,M,d\times I,\Om}|_{((V\ssm M)^d_\circ)^I}\ar[d]^-\fp\\
(\Div^{d,\circ}_{V\ssm M})^I & \ar[l]_-\al((V\ssm M)_\circ^d)^I.
}
\end{align*}
Henceforth, assume that $Y$ is geometrically connected over $k$. Now proper base change and \ref{ss:ICpullbacks} show that
\begin{align*}
\al^*R\fp_!(\sF^{(d)(I_1,\dotsc,I_k)}_{M,I,W_{\Om,H}}|_{(\Div^{d,\circ}_{V\ssm M})^I}) = \sH_{M,d\times I,W_{\Om,H}}|_{((V\ssm M)^d_\circ)^I},
\end{align*}
so Theorem \ref{ss:AGKRRV}.a) and \ref{ss:symmetrizedescend} enable us to identify $R\fp_!(\sF^{(d)(I_1,\dotsc,I_k)}_{M,I,W_{\Om,H}}|_{(\Div^{d,\circ}_{V\ssm M})^I})_{\ov{k}}$ with the image of $\sT_{M,I,W_{\Om,H}}^{(d)}|_{(\Div^{d,\circ}_{V\ssm M})^I_{\ov{k}}}$ in $\Shv((\Div^{d,\circ}_{V\ssm M})^I_{\ov{k}})$. Repeating the arguments in \ref{ss:shtukapf} yields an isomorphism
\begin{align*}
F_{(d)I_1}:\Frob_{I_1}^*R\fp_!(\sF^{(d)(I_2,\dotsc,I_k,I_1)}_{M,I,W_{\Om,H}}|_{(\Div^{d,\circ}_{V\ssm M})^I})\ra^\sim R\fp_!(\sF^{(d)(I_1,\dotsc,I_k)}_{M,I,W_{\Om,H}}|_{(\Div^{d,\circ}_{V\ssm M})^I}),
\end{align*}
and we see from Theorem \ref{ss:AGKRRV}.a) and \ref{ss:partialfrobeniuscompatibility} that, under our identification, the equivariance data of $\sT_{M,I,W_{\Om,H}}^{(d)}|_{(\Div^{d,\circ}_{V\ssm M})^I_{\ov{k}}}$ corresponds to the pullbacks $\ov{F}_{(d)i}$ of the $F_{(d)i}$ to $(\Div^{d,\circ}_{V\ssm M})^I_{\ov{k}}$.

The comparison with $\sT^{(d)}_{M,I,W_{\Om,H}}$ shows that $R\fp_!(\sF^{(d)(I_2,\dotsc,I_k,I_1)}_{M,I,W_{\Om,H}}|_{(\Div^{d,\circ}_{V\ssm M})^I})$ is independent up to isomorphism of the ordered partition $I_1,\dotsc,I_k$, so we denote this ind-(constructible complex of $\ov\bQ_\ell$-sheaves) on $(\Div^{d,\circ}_{V\ssm M})^I$ by $\sH_{M,I,W_{\Om,H}}^{(d)}$. As its pullback under the finite \'etale $\al$ has ind-smooth cohomology sheaves by Theorem \ref{ss:cohomologyissmooth}, we see that $\sH_{M,I,W_{\Om,H}}^{(d)}$ does as well.

\subsection{}\label{ss:maintheorem}
From here, we obtain the following generalization of Theorem A.
\begin{thm*}
The complex of intersection cohomology with compact support of $\Sht_{G,N,I,\Om}|_{Q_I}$ with coefficients in $\ov\bQ_\ell$ canonically lifts from an object of $D^b(\Weil(Q)^I,\ov\bQ_\ell)$ to an object of $D^b((\Weil(F)^d\rtimes\fS_d)^I,\ov\bQ_\ell)$.
\end{thm*}
\begin{proof}
Choose a geometric generic point $\ov{\eta_I}$ of $X^I$. Now our object of $D^b(\Weil(Q)^I,\ov\bQ_\ell)$ is the image of $\sT_{N,I,W_{\Om,G}}$ under the functor from \ref{ss:fiberfunctors}, where we use the finite-dimensionality of $\Sht_{G,N,I,\Om}|_{(U\ssm N)^I}$ and Theorem \ref{ss:AGKRRV}.a) to see that the image lies in $D^b(\Weil(Q)^I,\ov\bQ_\ell)$, and we use \ref{ss:ICpullbacks} and Theorem \ref{ss:cohomologyissmooth} to identify its underlying complex of $\ov\bQ_\ell$-vector spaces with $\sH_{N,I,W_{\Om,G},\ov{\eta_I}}$.

Choose a geometric generic point $\ov{\be_{d\times I}}$ of $Y^{d\times I}$. By projecting to $(\Div^d_Y)^I$, we view $\ov{\be_{d\times I}}$ as a geometric generic point of $(\Div^d_Y)^I$. Proper base change and \ref{ss:ICpullbacks} show that
  \begin{align*}
    (m^{-1})^{I,*}\sH^{(d)}_{M,I,W_{\Om,H}} = \sH_{N,I,W_{\Om,G}},
  \end{align*}
  so after choosing an \'etale path $\ov{\be_{d\times I}}\rightsquigarrow (m^{-1})^I(\ov{\eta_I})$, we obtain an isomorphism $\sH_{N,I,W_{\Om,G},\ov{\eta_I}}\ra^\sim\sH^{(d)}_{M,I,W_{\Om,H},\ov{\be_{d\times I}}}$ by \ref{ss:physicaldescent}.

  But the pullback of $\sH^{(d)}_{M,I,W_{\Om,H}}$ under $\al$ equals $\sH_{M,d\times I,W_{\Om,H}}|_{((V\ssm M)^d_\circ)^I}$, which identifies $\sH^{(d)}_{M,I,W_{\Om,H},\ov{\be_{d\times I}}}$ with $\sH_{M,d\times I,W_{\Om,H},\ov{\be_{d\times I}}}$. Theorem \ref{ss:cohomologyissmooth} allows us to view the latter as the image in $D^b((\Weil(F)^d\rtimes\fS_d)^I,\ov\bQ_\ell)$ of $\sT_{M,d\times I,W_{\Om,H}}$ under the functor from \ref{ss:fiberfunctors}. Finally, the commutative diagram
  \begin{align*}
    \xymatrix{((\Shv((V\ssm M)_{\ov{k}})^{B\bZ})^{\otimes(d\times I)})^{B\fS_d^I}\ar[r]\ar[d] & D((\Weil(F)^d\ltimes\fS_d)^I,\ov\bQ_\ell)\ar[d] \\ 
    (\Shv((U\ssm N)_{\ov{k}})^{B\bZ})^{\otimes I} \ar[r] & D(\Weil(Q)^I,\ov\bQ_\ell)}
  \end{align*}
shows that restricting along the $I$-fold product of $\Weil(Q)\hookrightarrow\Weil(F)^d\rtimes\fS_d$ recovers our original object of $D^b(\Weil(Q)^I,\ov\bQ_\ell)$, as desired.
\end{proof}

\subsection{}\label{ss:plecticalgebra}
We prove the following generalization of Theorem B by using the Hecke compatibility of our constructions.
\begin{thm*}
The action of $(\Weil(F)^d\rtimes\fS_d)^I$ from Theorem \ref{ss:maintheorem} on the level of cohomology groups commutes with the action of $\fH_{G,N}$ from \ref{ss:heckecorrespondence}.
\end{thm*}
\begin{proof}
Under the identification $G(\bA_Q)=H(\bA_F)$, we see that the compact open subgroup $K_{G,N}$ corresponds to $K_{H,M}$. This identifies $\fH_{G,N}$ with $\fH_{H,M}$. For any $g$ in $G(\bA_Q)$, write $N(g)$ for the finite set of closed points of $X$ where $g$ does not lie in $G(\cO_x)$, and write $M(g)$ for $m^{-1}(N(g))$. Remark \ref{rem:extension} gives an associated finite \'etale correspondence on $\Sht^{(I_1,\dotsc,I_k)}_{G,N,I,\Om}|_{(U\ssm(N\cup N(g)))^I}$ over $(U\ssm(N\cup N(g)))^I$. By viewing $g$ as an element of $H(\bA_F)$ instead, we obtain analogous correspondences on
\begin{align*}
\Sht^{(d)(I_1,\dotsc,I_k)}_{H,M,I,\Om}|_{(\Div^d_{V\ssm(M\cup M(g))})^I}\mbox{ and }\Sht^{(d\times I_1,\dotsc,d\times I_k)}_{H,M,d\times I,\Om}|_{(V\ssm(M\cup M(g)))^{d\times I}}.
\end{align*}
 Because these correspondences are finite \'etale, their pullbacks preserve intersection complexes, giving us a cohomological correspondence on $\sF^{(d)(I_1,\dotsc,I_k)}_{M,I,W_{\Om,H}}$.

Restricting \ref{ss:plecticsht} to $(\Div^d_{V\ssm(M\cup M(g))})^I$ yields a morphism
\begin{align*}
\Sht^{(I_1,\dotsc,I_k)}_{G,N,I,\Om}|_{(U\ssm(N\cup N(g)))^I}\ra\Sht^{(d)(I_1,\dotsc,I_k)}_{H,M,I,\Om}|_{(\Div^d_{V\ssm(M\cup M(g))})^I}.
\end{align*}
Note that our correspondence on $\Sht^{(I_1,\dotsc,I_k)}_{G,N,I,\Om}|_{(U\ssm(N\cup N(g)))^I}$ is precisely the pullback of our correspondence on $\Sht^{(d)(I_1,\dotsc,I_k)}_{H,M,I,\Om}|_{(\Div^d_{V\ssm(M\cup M(g))})^I}$ along the above morphism. Similarly, we see that our correspondence on
\begin{align*}
\Sht^{(d\times I_1,\dotsc,d\times I_k)}_{H,M,d\times I,\Om}|_{(V\ssm(M\cup M(g)))^{d\times I}}
\end{align*}
is the pullback of our correspondence on $\Sht^{(d)(I_1,\dotsc,I_k)}_{H,M,I,\Om}|_{(\Div^d_{V\ssm(M\cup M(g))})^I}$ along $\al$, up to universal homeomorphism.

The proof of Theorem \ref{ss:maintheorem} constructs the $(\Weil(F)^d\rtimes\fS_d)^I$-action on $\sH^p_{N,I,W_{\Om,G}}|_{\ov{\eta_I}}$ by identifying the latter with $\sH^p_{M,d\times I,W_{\Om,H}}|_{\ov{\be_{d\times I}}}$. The above shows that, under this identification, the $\fH_{G,N}$-action on $\sH^p_{N,I,W_{\Om,G}}|_{\ov{\eta_I}}$ coincides with the $\fH_{G,N}=\fH_{H,M}$-action on $\sH^p_{M,d\times I,W_{\Om,H}}|_{\ov{\be_{d\times I}}}$. From here, the desired commutativity follows from \ref{ss:physicaldescent}, \ref{ss:partialfrobeniuscompatibility}, and \ref{ss:adeliccompatible}.
\end{proof}

\subsection{}\label{ss:totallysplitfiber}
Before turning to Theorem C, we need some notation on the splitting behavior of points along $Y\ra X$. Let $k'$ be a finite extension of $k$ with degree $r$, and let $\ul{x}=(x_i)_{i\in I}$ be a point of $(U\ssm N)^I(k')$ such that every $x_i$ \emph{splits completely} in $V\ssm M$, i.e. the inverse image $m^{-1}(x_i)$ is a disjoint union of $d$ points $(y_{h,i})_{h=1}^d$ of $(V\ssm M)(k')$.

Because $(m^{-1})^I$ is a monomorphism, we get a Cartesian square
\begin{align*}
\xymatrix{\ul{x}\ar[r]^-{(m^{-1})^I}\ar[d] & (m^{-1})^I(\ul{x})\ar[d]\\
(U\ssm N)^I\ar[r]^-{(m^{-1})^I} & (\Div_{V\ssm M}^d)^I
}
\end{align*}
whose top arrow is an isomorphism. Therefore restricting \ref{ss:plecticsht} to $(m^{-1})^I(\ul{x})$ yields an isomorphism $\Sht^{(I_1,\dotsc,I_k)}_{G,N,I}|_{\ul{x}}\ra^\sim\Sht^{(d)(I_1,\dotsc,I_k)}_{H,M,I}|_{(m^{-1})^I(\ul{x})}$, and further restriction yields an isomorphism $\Sht^{(I_1,\dotsc,I_k)}_{G,N,I,\Om}|_{\ul{x}}\ra^\sim\Sht^{(d)(I_1,\dotsc,I_k)}_{H,M,I,\Om}|_{(m^{-1})^I(\ul{x})}$.

Because every $x_i$ splits completely in $V\ssm M$, we see that the preimage of $(m^{-1})^I(\ul{x})$ under $\al$ equals
\begin{align*}
\{(y_{\sg(h,i)})_{j\in[d],i\in I}\in(V\ssm M)^{d\times I}(k')\mid\sg\in\fS_d^I\},
\end{align*}
where each $(y_{\sg(h,i)})_{h\in[d],i\in I}$ maps isomorphically to $(m^{-1})^I(\ul{x})$ under $\al$. Write $\ul{y}$ for the point $(y_{h,i})_{h\in[d],i\in I}$ of $(V\ssm M)^{d\times I}(k')$ induced by our enumeration of the $m^{-1}(x_i)$. Then pulling back \ref{ss:shtukasymmetricaction} along $(m^{-1})^I(\ul{x})\lea^\sim\ul{y}$ yields an isomorphism
\begin{align*}
\Sht^{(d)(I_1,\dotsc,I_k)}_{H,M,I}|_{(m^{-1})^I(\ul{x})}\lea^\sim\Sht^{(d\times I_1,\dotsc,d\times I_k)}_{H,M,d\times I}|_{\ul{y}},
\end{align*}
and further restriction yields a universal homeomorphism
\begin{align*}
\Sht^{(d)(I_1,\dotsc,I_k)}_{H,M,I,\Om}|_{(m^{-1})^I(\ul{x})}\lea\Sht^{(d\times I_1,\dotsc,d\times I_k)}_{H,M,d\times I,\Om}|_{\ul{y}}.
\end{align*}

\subsection{}\label{ss:specialfibers}
From \ref{ss:totallysplitfiber}, we obtain the following implications for cohomology. Our zig-zag of universal homeomorphisms
\begin{align*}
\Sht^{(I_1,\dotsc,I_k)}_{G,N,I,\Om}|_{\ul{x}}\ra^\sim\Sht^{(d)(I_1,\dotsc,I_k)}_{H,M,I,\Om}|_{(m^{-1})^I(\ul{x})}\lea\Sht^{(d\times I_1,\dotsc,d\times I_k)}_{H,M,d\times I,\Om}|_{\ul{y}}
\end{align*}
allows us to identify the \'etale cohomology of the left and right terms.

Choose a geometric point $\ov{\ul{x}}$ lying over $\ul{x}$. Now $\sF^{(I_1,\dotsc,I_k)}_{N,I,W_{\Om,G}}$ restricts to the intersection complex of $\Sht^{(I_1,\dotsc,I_k)}_{G,N,I,\Om}|_{\ul{x}}$, so proper base change indicates that $\sH^p_{N,I,W_{\Om,G}}|_{\ov{\ul{x}}}$ is naturally isomorphic to the $p$-th intersection cohomology group with compact support of $\Sht^{(I_1,\dotsc,I_k)}_{G,N,I,\Om}|_{\ov{\ul{x}}}$ with coefficients in $\ov\bQ_\ell$. Analogous statements hold for $\Sht^{(d)(I_1,\dotsc,I_k)}_{H,M,I,\Om}|_{(m^{-1})^I(\ul{x})}$ and $\Sht^{(d\times I_1,\dotsc,d\times I_k)}_{H,M,d\times I,\Om}|_{\ul{y}}$, in a manner compatible with the identifications induced by the above zig-zag.

\subsection{}\label{ss:plecticfrob}
We now prove the following generalization of Theorem C. After choosing an \'etale path $\ov{\eta_I}\rightsquigarrow\ov{\ul{x}}$, Theorem \ref{ss:cohomologyissmooth} yields a specialization isomorphism
\begin{align*}
\sH^p_{N,I,W_{\Om,G}}|_{\ov{\ul{x}}}\ra^\sim\sH^p_{N,I,W_{\Om,G}}|_{\ov{\eta_I}}
\end{align*}
Applying $(m^{-1})^I$ to $\ov{\eta_I}\rightsquigarrow\ov{\ul{x}}$ induces an \'etale path $(m^{-1})^I(\ov{\eta_I})\rightsquigarrow (m^{-1})^I(\ov{\ul{x}})$, and composing this with our \'etale path $\ov{\be_{d\times I}}\rightsquigarrow (m^{-1})^I(\ov{\eta_I})$ yields an \'etale path $\ov{\be_{d\times I}}\rightsquigarrow (m^{-1})^I(\ov{\ul{x}})$. The isomorphism $(m^{-1})^I(\ul{x})\lea^\sim\ul{y}$ enables us to view $\ov{\ul{x}}$ as a geometric point $\ov{\ul{y}}$ lying over $\ul{y}$, so this amounts to an \'etale path $\ov{\be_{d\times I}}\rightsquigarrow\al(\ov{\ul{y}})$. Because $\al$ is \'etale at $\be_{d\times I}$ and $\ul{y}$, this corresponds to an \'etale path $\ov{\be_{d\times I}}\rightsquigarrow\ov{\ul{y}}$. Applying Theorem \ref{ss:cohomologyissmooth} again yields another specialization isomorphism
\begin{align*}
\sH^p_{M,d\times I,W_{\Om,H}}|_{\ov{\ul{y}}}\ra^\sim\sH^p_{M,d\times I,W_{\Om,H}}|_{\ov{\be_{d\times I}}}.
\end{align*}

Let $(h,i)$ be in $d\times I$. From our \'etale path $\ov{\be_{d\times I}}\rightsquigarrow\ov{\ul{y}}$, we obtain an element $\ga_{y_{h,i}}$ of $\Weil(V\ssm M)$ as in \ref{ss:partialfrobeniusgaloisfrobenius}. Since $\Frob_{(h,i)}^r$ fixes $\ul{y}$, we see that
\begin{align*}
F_{(h,i)}\circ\Frob_{(h,i)}^*(F_{(h,i)})\circ\dotsb\circ\Frob_{(h,i)}^{r-1,*}(F_{(h,i)})
\end{align*}
restricts to an automorphism of $\sH^p_{M,d\times I,W_{\Om,H}}|_{\ov{\ul{y}}}$.
\begin{thm*}
Under the identifications
\begin{align*}
\sH^p_{N,I,W_{\Om,G}}|_{\ov{\eta_I}}\ra^\sim\sH^p_{M,d\times I,W_{\Om,H}}|_{\ov{\be_{d\times I}}}\lea^\sim\sH^p_{M,d\times I,W_{\Om,H}}|_{\ov{\ul{y}}},
\end{align*}
 this automorphism corresponds to the action of $\ga_{y_{h,i}}$ in the $(h,i)$-th entry of $\Weil(V\ssm M)^{d\times I}$.
\end{thm*}
\begin{proof}
The proof of Theorem \ref{ss:maintheorem} constructs the action of $\Weil(V\ssm M)^{d\times I}$ on $\sH^p_{N,I,W_{\Om,G}}|_{\ov{\eta_I}}$ via the isomorphism $\sH^p_{N,I,W_{\Om,G}}|_{\ov{\eta_I}}\ra^\sim\sH^p_{M,d\times I,W_{\Om,H}}|_{\ov{\be_{d\times I}}}$. So it suffices to see that this corresponds to the action of $\ga_{y_{h,i}}$ in the $(h,i)$-th entry of $\Weil(V\ssm M)^{d\times I}$ on $\sH^p_{M,d\times I,W_{\Om,H}}|_{\ov{\be_{d\times I}}}$, which follows immediately from Proposition \ref{ss:galfrob}.
\end{proof}

\subsection{}
When $r=1$, the action from Theorem \ref{ss:plecticfrob} has the following description in terms of the special fiber of the moduli space of shtukas. First, recall from \ref{ss:eps} that we can use any ordered partition of $d\times I$ to compute $\sH^p_{M,d\times I,W_{\Om,H}}$. Let $P$ be any ordered partition of $d\times I\ssm (h,i)$. Then $F_{(h,i)}$ is induced by the morphism $\Fr^{((h,i),P)}_{(h,i),M,d\times I}$, so we see from \ref{ss:specialfibers} that its action on $\sH^p_{M,d\times I,W_{\Om,H}}$ is induced by the restriction
\begin{align*}
\Fr^{((h,i),P)}_{(h,i),M,d\times I}:\Sht^{((h,i),P)}_{H,M,d\times I,\Om}|_{\ul{y}}\ra\Sht^{(P,(h,i))}_{H,M,d\times I,\Om}|_{\ul{y}}.
\end{align*}

\subsection{}
The case of general $r$ is more complicated for the following reason. We use Theorem \ref{ss:geometricsatake}.b) to identify $R\fp_!(\sF^{((h,i),P)}_{M,d\times I,W_{\Om,H}})$ with $R\fp_!(\sF^{(P,(h,i))}_{M,d\times I,W_{\Om,H}})$, which is what enables us to iteratively compose $F_{(h,i)}$. This identification does not seem to arise from an explicit cohomological correspondence, so for general $r$ and $\ul{x}$ this impedes us from similarly describing the action of
\begin{align*}
F_{(h,i)}\circ\Frob_{(h,i)}^*(F_{(h,i)})\circ\dotsb\circ\Frob_{(h,i)}^{r-1,*}(F_{(h,i)})
\end{align*}
on $\sH^p_{M,d\times I,W_{\Om,H}}|_{\ov{\ul{y}}}$. However, when all the $x_i$ are disjoint, the $y_{h,i}$ will also be disjoint, so we can use \ref{ss:grfactorization} to explicitly identify $\Sht^{((h,i),P)}_{H,M,d\times I,\Om}|_{\ul{y}}$ with $\Sht^{(P,(h,i))}_{H,M,d\times I,\Om}|_{\ul{y}}$ via pulling back along $\de\circ\ga$. Therefore we obtain a description of the action from Theorem \ref{ss:plecticfrob} in terms of the special fiber of the moduli space of shtukas in this case.

\bibliographystyle{habbrv}
\bibliography{biblio}
\end{document}